\documentclass[11pt,a4paper]{amsart}
\usepackage{anysize} 

\usepackage{amsthm}  
\usepackage{amsmath}
\usepackage{amsfonts}
\usepackage{amssymb}

\usepackage{multicol} 
\usepackage{graphicx} 
\usepackage[dvipsnames]{xcolor} 
\usepackage{pgf,tikz}
\usepackage{tikz-cd}
\usepackage{mathrsfs}
\usepackage{enumerate}
\usepackage{verbatim} 
\usepackage{array}
\usepackage{hyperref}

\hypersetup{
	colorlinks,
	linkcolor={blue!60!black},
	citecolor={blue!60!black},
	urlcolor={blue!60!black}
}

\usepackage{centernot}
\usepackage{mathtools} 
\usepackage{enumitem} 
\usepackage{mathdots} 
\usepackage{stmaryrd} 
\usetikzlibrary{arrows}

\marginsize{2 cm}{2 cm}{2 cm}{2 cm}

\numberwithin{equation}{section}
\newtheorem{lemma}[equation]{Lemma}
\newtheorem{proposition}[equation]{Proposition}
\newtheorem{theorem}[equation]{Theorem}
\newtheorem{corollary}[equation]{Corollary}

\theoremstyle{definition}
\newtheorem{defprop}[equation]{Definition-Proposition}

\newtheorem{remark}[equation]{Remark}
\newtheorem{definition}[equation]{Definition}
\newtheorem{construction}[equation]{Construction}

\newtheorem{notation}[equation]{Notation}
\numberwithin{equation}{section}

\newcommand{\C}{{\mathbb C}}
\newcommand{\bD}{{\mathbb D}}
\newcommand{\bH}{{\mathbb H}}
\newcommand{\N}{{\mathbb N}}
\newcommand{\Q}{{\mathbb Q}}
\newcommand{\R}{{\mathbb R}}
\newcommand{\Z}{{\mathbb Z}}
\newcommand{\cL}{{\mathcal L}}
\newcommand{\cV}{{\mathcal V}}
\newcommand{\oL}{\ov{\cL}}
\newcommand{\cF}{{\mathcal F}}
\newcommand{\sC}{{\mathscr C}}
\newcommand{\sK}{{\mathscr K}}
\newcommand{\cK}{{\mathcal K}}
\newcommand{\fm}{{\mathfrak m}}

\newcommand{\ul}[1]{\underline{#1}}
\newcommand{\wt}[1]{\widetilde{#1}}
\newcommand{\ov}[1]{\overline{#1}}

\DeclareMathOperator{\Homm}{\mathcal{H}om}

\DeclareMathOperator{\Cone}{Cone}
\DeclareMathOperator{\Res}{Res}
\DeclareMathOperator{\res}{res}
\DeclareMathOperator{\Spec}{Spec}
\DeclareMathOperator{\coker}{coker}
\DeclareMathOperator{\Tot}{Tot}
\DeclareMathOperator{\Tors}{Tors}
\DeclareMathOperator{\Ext}{Ext}
\DeclareMathOperator{\Hom}{Hom}
\DeclareMathOperator{\Id}{Id}
\DeclareMathOperator{\Sym}{Sym}

\newcommand{\Alex}{H_i(U^f, \Q)}
\newcommand{\TorsH}{\Tors_R \Alex}
\newcommand{\pairing}{\langle\cdot{,}\cdot\rangle}

\author{Eva Elduque}
\address{Departamento de Matem\' aticas, Universidad Aut\' onoma de Madrid and ICMAT, 28049 Madrid, Spain}
\urladdr{https://matematicas.uam.es/~eva.elduque/}
\email{eva.elduque@uam.es}
\author{Mois\'es Herrad\'on Cueto}
\address{Departamento de Matem\' aticas, Universidad Aut\' onoma de Madrid, 28049 Madrid, Spain}
\urladdr{https://matematicas.uam.es/~moises.herradon/}
\email{moises.herradon@uam.es}

\keywords{infinite cyclic cover, Alexander module, mixed Hodge structure, thickened complex}
\subjclass[2020]{14C30, 14D07, 14F40, 14F45, 32S35, 32S40, 32S50, 32S55, 55N25,
 55N30}
\thanks{E. Elduque and M. Herradón Cueto 
are partially supported by the Grant PID2022-138916NB-I00 funded by MCIN/AEI/ 10.13039/501100011033 and by ERDF A way of making Europe. E. Elduque is also partially supported by the Ramón y Cajal Grant RYC2021-031526-I funded by MCIN/AEI /10.13039/501100011033 and by the European Union NextGenerationEU/PRTR}

\title{Compatibility of Hodge theory on Alexander modules}

\begin{document}

\begin{abstract}
  Let \(U\) be a smooth connected complex algebraic variety, and let \(f\colon U\to \C^*\) be an algebraic map. To the pair \((U,f)\) one can associate an infinite cyclic cover \(U^f\), and (homology) Alexander modules are defined as the homology groups of this cover. In two recent works, the first of which is joint with Geske, Maxim and Wang, we developed two different ways to put a mixed Hodge structure on Alexander modules. Since they are not finite dimensional in general, each approach replaces the Alexander module by a different finite dimensional module: one of them takes the torsion submodule, the other takes finite dimensional quotients, and the constructions are not directly comparable. In this note, we show that both constructions are compatible, in the sense that the map from the torsion to the quotients is a mixed Hodge structure morphism.
\end{abstract}

\maketitle

\begin{center}\textit{Dedicated to Lauren\c{t}iu Paunescu and Alexander I. Suciu on the occasion of their 70\textsuperscript{th} birthdays.}\end{center}

	\tableofcontents

  \section{Introduction}

  The goal of this note is to establish the relation between two recent constructions of mixed Hodge structures (MHS) on Alexander modules. Let us start by defining Alexander modules. Let \(U\) be a smooth connected complex algebraic variety, and let \(f\colon U\to \C^*\) be an algebraic map. From this information, one can construct an infinite cyclic cover of \(U\), that we denote \(U^f\), by pulling back the exponential map \(\exp\colon \C\to \C^*\), i.e. as the following fiber product:
  \begin{equation}\label{eq:fiberProductIntro}
  \begin{tikzcd}
    U^f \subset U\times \C\arrow[r, "\wt f"]
    \arrow[dr, phantom, very near start,"\lrcorner"]
    \arrow[d, "\pi"]
    &
    \C\arrow[d, "\exp"]
    \\
    U\arrow[r, "f"]
    &
    \C^*.
  \end{tikzcd}
  \end{equation}
 The exponential is a regular cover with deck group \(\pi_1(\C^*)\cong\Z\). Therefore, so is its pullback \(\pi\colon U^f\to U\). The (one variable, homological, with rational coefficients) \emph{Alexander modules} of the pair \((U, f)\) are the homology groups \(H_i(U^f, \Q)\) for \(i\ge 0\). They naturally have an automorphism, given by the effect in homology of the counterclockwise generator of $\pi_1(\C^*)$, seen as a deck transformation. This gives them the structure of a \(\Q[t^{\pm 1}]\)-module.

One motivation for this construction is the following analogy: suppose that \(f\colon U\to S^1\) is a locally trivial fibration. Then, an infinite cyclic cover \(U^f\to U\) can be constructed by pulling back the universal covering map \(\R\to S^1\). In this case, the map \(\wt f: U^f\to \R\) appearing in the pullback diagram is a locally trivial fibration over a contractible base, so \(U^f\) is homotopy equivalent to the fiber of the map \(\wt f\), or equivalently, to the fiber of \(f\). Therefore, \(H_i(U^f, \Q)\) will compute the homology of the fiber. In our situation, \(f\) is an algebraic map to $\C^*\simeq S^1$, and it need not be a locally trivial fibration, but it may be thought of as a generalization of the homology of the fiber.

  A well-known example is the following situation: let \(f\colon \C^n\to \C\) be a non-constant polynomial, cutting out a hypersurface \(H\subset \C^n\). Let \(U = \C^n\setminus H\), so that \(f\) restricts to a map \(U\to \C^*\). Alexander modules have been long studied in this context, e.g. \cite{DL,DimcaNemethi,Libgober94,Liu,Max06}.

  This paper studies mixed Hodge structures on the Alexander modules. The fundamental theorem concerning mixed Hodge structures, due to Deligne \cite{DeligneII}, states that the (co)homology of complex algebraic varieties has a mixed Hodge structure. The coverings \(U^f\) that we are concerned with are not algebraic varieties in general, and the corresponding Alexander modules are finitely generated \(R\)-modules, but  not necessarily finite dimensional over \(\Q\). As a result, to develop a Hodge theory for the homology of $U^f$, one needs to extract meaningful finite dimensional $\Q$-vector spaces from the Alexander modules. We outline two main strategies to do this:

\begin{enumerate}
  \item\label{intropart:mhsalexander} Restricting our focus to the largest possible finite dimensional submodule, namely, the torsion submodule \(\Tors_R H_i(U^f, \Q)\). This is the approach taken by Christian Geske, Lauren\c{t}iu Maxim, Botong Wang and the authors in \cite{mhsalexander}. In loc. cit. a canonical and functorial MHS is constructed on \(\Tors_R H_i(U^f, \Q)\), and it is shown to satisfy many desirable properties, many of which resemble those of the fiber of \(f\).
  \item\label{intropart:abeliancovers} Restricting instead to finite dimensional quotient modules. More recently, in \cite{abeliancovers}, the authors have constructed a MHS on every quotient of \(H_i(U^f, \Q)\) of the form \(\displaystyle \frac{H_i(U^f, \Q)}{(t^N-1)^m H_i(U^f, \Q)}\), for \(m,N\ge 1\). Furthermore, they are compatible via the natural quotient maps between them (varying both $m$ and $N$), so one may even consider their inverse limit when $m\to\infty$. Note that \(H_i(U^f, \Q)\) is a finitely generated module over the principal ideal domain (PID) \(R\), so it is the direct sum of its torsion and a free part. It follows from the results of \cite{BudurLiuWang} that the torsion part is annihilated by \((t^N-1)^m\) for some \(m\) and \(N\), so these quotients contain more information than $\Tors_R H_i(U^f,\Q)$ in the following sense: the Alexander module embeds into the inverse limit of these quotients.

  In fact, the construction in \cite{abeliancovers} is much more general, and it allows one to replace \(\C^*\) with any semiabelian variety \(G\), where its universal cover takes the role of the exponential map, and the homology groups become groups over the group ring of \(\pi_1(G)\cong \Z^r\) (i.e. a Laurent polynomial ring). In place of \((t^N-1)^m\), one may take the \(m\)-th power of the ideal \(( g-1\mid g\in H)\), where \(H\) can be any finite index subgroup of \(\pi_1(G)\).
  \end{enumerate}

  Regarding the approach in \eqref{intropart:abeliancovers}, and taking a geometric point of view on the module structure, \(\Spec \Q[\pi_1(G)]\) is the character variety of \(G\), i.e. the space parametrizing rank 1 $\Q$-local systems. Let \(\fm\) be the augmentation ideal of \(G\), i.e. \((g-1\mid g\in G)\). Then, taking the inverse limit \(\varprojlim_m \Q[\pi_1(G)] / \fm^m\) corresponds to restricting to a formal neighborhood of the identity. Alexander modules are quasicoherent sheaves on the character variety, and the inverse limit \(\varprojlim_m \frac{H_i(U^f, \Q)}{\fm^m H_i(U^f, \Q)}\) is the restriction of this sheaf to the formal neighborhood of the identity. This is reminiscent of the study of resonance varieties, as in \cite{PapaSuciu,suciuSurvey}: it is a module from which one can construct the restriction of the Alexander varieties to the formal neighborhood of the identity. In \cite{suciuSurvey} (see also \cite{PapadimaSuciu10}), the relation is explained in detail.

  The two constructions outlined in \eqref{intropart:mhsalexander} and \eqref{intropart:abeliancovers} are similar, but independent. This raises the question of their relationship, which this paper aims to answer. We prove the following theorem (Corollary~\ref{cor:mainN}).

  \begin{theorem}\label{thm:intro}
    Let \(U\) be a smooth connected complex algebraic variety with an algebraic map \(f\colon U\to \C^*\), let \(U^f\) be the infinite cyclic covering map in Diagram~\eqref{eq:fiberProductIntro}. Let \(m, N\ge 1\). Let \(R = \Q[\pi_1(\C^*)]\cong \Q[t^{\pm 1}]\). Consider the MHS on \(\Tors_R H_i(U^f,\Q)\) mentioned in~\eqref{intropart:mhsalexander} (from \cite{mhsalexander}), and the MHS on\linebreak \(\frac{H_i(U^f, \Q)}{(t^N-1)^m H_i(U^f, \Q)}\) mentioned in~\eqref{intropart:abeliancovers} (from \cite{abeliancovers}). Since they are a submodule and a quotient of the Alexander module, respectively, there is a natural composition map:
  \[
    \Tors_R H_i(U^f, \Q) \hookrightarrow H_i(U^f, \Q) \twoheadrightarrow \frac{H_i(U^f, \Q)}{(t^N-1)^mH_i(U^f, \Q)}.
  \]
  This composition map is a MHS morphism for all $m, N\geq 1$. Moreover, there exists $N\in \Z_{\geq 1}$ such that this composition map is injective for $m\gg 1$.
  \end{theorem}

  To make sense of the theorem, let us comment on the difference between the constructions of both MHS. The main tool in both cases are Deligne's mixed Hodge complexes of sheaves, which endow their hypercohomology spaces with MHS. Both constructions use a mixed Hodge complex of sheaves that we call ``thickening'', which is obtained as a deformation of a known mixed Hodge complex of sheaves. These were constructed in \cite{mhsalexander}, and later vastly generalized in \cite{abeliancovers}. Both constructions use the same mixed Hodge complexes to perform a thickening, although the corresponding thickenings and filtrations defined therein are slightly different.

  The first hurdle is that the theory of mixed Hodge complexes of sheaves is well-suited for cohomology, so the problem must be translated to cohomology in order to use these tools.

  The cohomology groups \(H^i(U^f, \Q)\) are very badly behaved.  They are the dual spaces to \(H_i(U^f, \Q)\), which might be infinite dimensional. Therefore, the cohomology groups can be uncountably dimensional. The more manageable approach is to dualize over \(R\), by using local systems. There is a natural local system of \(R\)-modules \(\cL\) (see Definition \ref{def:cL}) such that \(H_i(U,\cL)\cong H_i(U^f, \Q)\). If we let \(\oL\) be the \(R\)-dual local system, we are interested in the cohomology of \(H^i(U,\oL)\).

  When it comes to the question of how to relate the homology Alexander modules to the cohomology of \(\oL\), parts~\eqref{intropart:mhsalexander} and~\eqref{intropart:abeliancovers} take very different approaches. The approach in~\eqref{intropart:mhsalexander} is roughly to dualize over \(R\), and to use the Universal Coefficient Theorem over a PID to obtain isomorphisms
  \begin{equation}\label{eq:UCTResIntro}
    \Hom_\Q(\Tors_R H_i(U, \cL), \Q) \cong \Tors_R H^{i+1}(U, \oL).
  \end{equation}
  These isomorphisms are difficult to work with: they rely on an isomorphism that exists for a finite dimensional \(\Q[t^{\pm 1}]\)-module \(A\), \(\Ext^1_R(A, R)\cong \Hom_\Q(A, \Q)\). This is of course false for general finitely generated modules, and it is not clear how it can lift to the worlds of sheaves and/or triangulated categories, two of the most useful tools at our disposal.
    
  Part~\eqref{intropart:abeliancovers} takes a different approach, which is to take finite dimensional quotients of \(\cL\) and \(\oL\), and relate the (co)homology of these using the Universal Coefficient Theorem over \(\Q\). Therefore, to relate both structures, one must first understand how these two dualities are related.

  Our goal is to understand the relation between these approaches at the level of sheaves. However, notice that there is a shift in equation~\eqref{eq:UCTResIntro}, which does not appear in the duality over \(\Q\). Therefore, if there are morphisms of sheaves that realize these identities, they must have a cohomological shift, and therefore they will be morphisms in the derived category.

  The main technique used in this paper to overcome this hurdle is to find objects in the derived category (i.e. complexes of sheaves) that interpolate between these two approaches. Then we must give these complexes of sheaves the structure of a mixed Hodge complex. This would be an arduous task, as their definition is very involved. However, a combination of properties of mixed Hodge complexes in previous work and of well-known constructions, such as the mixed cone (see Definition~\ref{def:mixedCone}) make this task much simpler. In fact, we do not explicitly use the definition of a mixed Hodge complex of sheaves in this note.

  The simpler duality used in the newer construction in \cite{abeliancovers} makes some of the proofs of the key results therein much simpler than the proofs of the analogous results in our initial approach in \cite{mhsalexander}, despite the more general setting.  Theorem~\ref{thm:intro} bridges the gap between those constructions, and allows one to apply the newer techniques and constructions to the original MHS of \cite{mhsalexander}. Conversely, there are properties proved in~\cite{mhsalexander} that are not generalized in the newer construction. Notably, \cite[Theorem 1.8]{mhsalexander} shows that the MHS can be compared to the limit MHS on the nearby cycles, which is specific to \(\C^*\), so no analogous theorem is proved in \cite{abeliancovers}. Further, several structure results are proved, such as a bound on the possible nonzero weights. As a corollary of Theorem~\ref{thm:intro}, these results all apply to the MHS in \cite{abeliancovers} for the case when the semiabelian variety $G$ is $\C^*$.

There had been other prior approaches to the construction of MHS on the torsion part of Alexander modules in some particular cases \cite{DL,HainHomotopy,KK,Lib96,Liu}, see the introduction of \cite{mhsalexander} for a comparison. In \cite{mhsalexander}, it is shown that the MHS constructed therein agrees with the one in \cite{DL}, which applies to the particular case of complements of hypersurfaces transversal at infinity, and which was  recovered by different techniques in \cite{Liu}. Therefore, as a corollary to Theorem~\ref{thm:intro}, we can also conclude that the MHS of \cite{DL, Liu} are compatible with the one in \cite{abeliancovers}.

  \subsection*{Outline of the paper}

  In Section~\ref{sec:prelims} we give an overview of previous results that will be necessary for the proof, and establish the notation for the rest of the paper. In Section~\ref{sec:previousMHS}, we state some properties of mixed Hodge complexes, as well as some details of the constructions in \cite{mhsalexander} and \cite{abeliancovers}. Section~\ref{sec:aux} contains the construction of mixed Hodge complexes and mixed Hodge structures that will be used to interpolate between both constructions. Finally, in Section~\ref{sec:main} we prove the theorem.

  \subsection*{Acknowledgments} We would like to thank our collaborators  Christian Geske, Lauren\c{t}iu Maxim and Botong Wang for the many interesting and fruitful discussions we have had regarding this topic, both during the preparation of our joint work \cite{mhsalexander} and afterwards. We would also like to thank Alex Suciu for the interesting conversations during the past summer.

  \section{Preliminaries and notation}\label{sec:prelims}

\subsection{Alexander modules}

\begin{definition}\label{def:Uf}
    Throughout the whole note, \(U\) will denote a smooth connected complex algebraic variety, and \(f\colon U \to \C^*\) will be an algebraic map. The \textbf{infinite cyclic cover} associated to \(f\) is the pullback of \(\exp\colon \C\to \C^*\) by \(f\), i.e. the covering \((U^f, \pi)\) in the following diagram:
  \[
  \begin{tikzcd}[column sep = 5em, row sep = 1.5em]
      U^f \coloneqq \{(x,z)\in U\times \C \mid f(x) = e^z \}
      \arrow[r, "\wt f"]
      \arrow[d, "\pi"]
      &
      \C\arrow[d, "\exp"]
      \\
      U\arrow[r, "f"]
      &
      \C^*.
  \end{tikzcd}
  \]
  It will always be denoted as \(U^f\).
\end{definition}

Since \(\exp\) is a regular covering with deck group \(\pi_1(\C^*)\cong \Z\), so is its pullback \(\pi\). Namely, the counterclockwise generator of \(\pi_1(\C^*)\) acts as \((x,z)\mapsto (x, z+2\pi i)\).

\begin{definition}
    Let \(i\ge 0\). The \(i\)-th (univariable, homological) \textbf{Alexander module} of \((U,f)\) (with coefficients in \(\Q\)) is \(
    H_i(U^f, \Q).
    \)
\end{definition}

\begin{notation}
  We will let \(R = \Q[\pi_1(\C^*)]\). If we let \(t\) be the counterclockwise generator of \(\pi_1(\C^*)\), then we obtain an isomorphism \(R\cong \Q[t^{\pm 1}] \). We will identify $R$ and $\Q[t^{\pm 1}]$ throughout the paper in this way.
\end{notation}

\begin{remark}
    Since \(U^f\) has an action of \(\pi_1(\C^*) \cong \Z\), so do the Alexander modules. Therefore, they have the structure of an \(R\)-module. Since algebraic varieties have the homotopy type of a finite CW complex, Alexander modules are finitely generated \(R\)-modules.
\end{remark}

  \subsection{The group ring and its quotients}

We will define some notation related to \(R\) that we use throughout this paper.

\begin{notation}
  For \(m\in \Z_{\ge 1}\), we define the rings \(R_\infty\) and \(R_m\) by
  \begin{align*}
    R_\infty \coloneqq \prod_{j=0}^\infty \Sym^j H_1(\C^*,\Q);&& R_m \coloneqq \frac{R_\infty}{\prod_{j=m}^\infty \Sym^j H_1(\C^*,\Q)}.
  \end{align*}
  If we let \(s\) be the counterclockwise generator of \(H_1(\C^*, \Q)\), then we obtain isomorphisms \(R_\infty\cong \Q\llbracket s \rrbracket\) and \(R_m \cong R_\infty / s^m R_\infty\). We define \(R_{-m}\) by
  \[
    R_{-m} \coloneqq \Hom_\Q(R_m, \Q).
  \]
  Multiplication on \(R_m\) gives \(R_{-m}\) the structure of an \(R_\infty\)-module.
\end{notation}

\begin{definition}\label{def:exp}
  We consider the following ring monomorphism:
  \[
    \begin{array}{rcl}
      R & \longrightarrow & R_\infty \\
      t & \longrightarrow & e^s = \displaystyle\sum_{j = 0}^\infty  \frac{s^j}{j!}.
    \end{array}
  \]
  Throughout this paper, \(R_\infty\) and every other \(R_\infty\)-module is seen as an \(R\)-module via this homomorphism.
\end{definition}

\begin{remark}\label{rem:RmDual}
  For all \(m\geq 1\), there is a canonical \(R_\infty\)-linear isomorphism \(R_{-m} = \Hom_\Q(R_m, \Q)\cong \frac{s^{-m}R_\infty}{R_\infty} \), given by this perfect pairing:
  \[
    \begin{array}{rrcl}
      \langle\cdot, \cdot\rangle_m \colon
      &
      \displaystyle \frac{s^{-m} R_\infty}{R_\infty} \times R_m
      &
      \longrightarrow
      &
      \Q
      \\
      &
      (g(s), \alpha(s))
      &
      \longmapsto
      &
      \res_0 (g\alpha),
    \end{array}
  \]
  where $\res_0$ denotes the residue at the point $s=0$ (i.e. the coefficient of \(s^{-1}\), or equivalently, for every Laurent polynomial $h(s)$  in $s$, \(\res_0 h(s)\) is the residue of \(h(s)ds\)).
\end{remark}

\begin{remark}\label{rem:comparisonResidues}
    If we let \(K\) be the field of fractions of \(R\), one can define the residue \(\res \colon K \to \Q\) as the sum of residues over every pole (i.e. at \(b\in \C\), one considers the coefficient of \((t-b)^{-1}\)). Note that the homomorphism from Definition~\ref{def:exp} induces an isomorphism \(\frac{(t-1)^{-m}R}{R} \cong \frac{s^{-m}R_\infty}{R_\infty}\), and furthermore, that both residues agree.
    \end{remark}
    
\begin{remark}
  Under the isomorphism of Remark~\ref{rem:RmDual}, the dual of the projection \(R_{m_1}\twoheadrightarrow R_{m_2}\) for \(m_1\ge m_2 > 0\) is the inclusion \(\displaystyle \frac{s^{-m_2}R_\infty}{R_\infty}\hookrightarrow \frac{s^{-m_1}R_\infty}{R_\infty}\).
\end{remark}

\begin{remark}\label{rem:RmDualMultiplication}
Let \(\{1^\vee, s^\vee, \ldots , (s^{m-1})^\vee\}\) be the basis of \(R_{-m}\) that is dual to \(\{1,s,\ldots , s^m\}\). The isomorphism of Remark~\ref{rem:RmDual} sends \((s^i)^\vee\in R_{-m}\) to \(s^{-1-i}\in\frac{s^{-m}R_\infty}{R_\infty}\). This follows from checking that \(\langle s^{j}, s^i\rangle_m = \res_0 s^{j+i}= 1\) if \(j+i=-1\) and vanishes otherwise. In \cite[Example 2.60]{abeliancovers}, an isomorphism \(R_m\cong R_{-m}\) is described, sending \((s^i)^\vee\) to \(s^{m-1-i}\). Therefore, the composition with this isomorphism is multiplication by \(s^{m}\): \(\frac{s^{-m}R_\infty}{R_\infty} \xrightarrow{\cong} R_{-m} \xrightarrow{\cong} R_m\). \end{remark}

  \subsection{Alexander modules and local systems}

In order to make computations about Alexander modules, we will regard them as the homology of a certain local system, or the cohomology of its dual. Here, we will define it and recall some of its properties.

\begin{definition}\label{def:cL}
    Let \(U, f, U^f, \pi\) be as in Definition \ref{def:Uf}. We define \(\cL = \pi_! \ul\Q_{U^f}\), as a sheaf in \(U\).
\end{definition}

\begin{remark}
    Since \(\pi\) is a covering map, it is a locally constant sheaf, i.e. a local system. On a simply connected neighborhood of any \(x\in U\), it is isomorphic to \(\bigoplus_{x'\in f^{-1}(x)} \ul\Q\) (note that this would be a direct product if we used the pushforward \(\pi_*\ul\Q_{U^f}\) instead).

    Furthermore, on such a neighborhood, this direct sum decomposition provides a basis on which the deck group \(\pi_1(\C^*)\) acts transitively and freely. Therefore, \(\cL\) is locally isomorphic to \(R\), and globally it is a local system of rank 1 free \(R\)-modules.
\end{remark}

\begin{proposition}\label{prop:UfvscL}
    Let \(i\ge 0\). There is a canonical \(R\)-module isomorphism \(H_i(U^f, \Q)\cong H_i(U, \cL)\).
\end{proposition}
\begin{proof}
    This follows from writing the complexes that compute both homology groups, and noticing that they are the same. For the definition of the chain complex computing local system homology, see \cite[Section 2.5]{dimca2004sheaves} or Proposition \ref{prop:chainHomology} below.
\end{proof}

\begin{definition}[{\cite[Remark 2.25]{abeliancovers}}]\label{def:oL}
    Let \(U,f,U^f\) be as in Definition \ref{def:Uf}, and \(\cL\) as in Definition \ref{def:cL}. Throughout, we will denote by \(\oL = \Homm_R(\cL, \ul R)\) its \(R\)-dual local system.
\end{definition}

\begin{remark}
    In fact, there is a canonical isomorphism of \(\Q\)-local systems \(\cL\cong \oL\). It does not preserve the \(R\)-module structure, but rather the generator \(t\in R\) acts on \(\oL\) as \(t^{-1}\) acts on \(\cL\). We will not use this fact here. For a proof, see \cite[Remark 2.16]{abeliancovers}.
\end{remark}

 \begin{remark}\label{rem:RmcLDual}
     For every \(m\in \Z\setminus\{0\}\), the perfect pairing
      \begin{equation}\label{eq:pairingTruncated}
      \begin{tikzcd}[row sep = 0.2em, column sep =0]
        \left(R_{-m}\otimes_R \oL\right)
        &
        \times
      &
      \left(R_m\otimes_R \cL\right)
      \arrow[rr]
      &&
      \ul\Q
      \\
      \quad\quad(\phi \otimes a
      &
      {,}
      &
      \alpha\otimes b)\quad
      \arrow[rr, mapsto]
      &\phantom{a}&
      \phi(\alpha\cdot a(b)).
    \end{tikzcd} 
    \end{equation}
    induces an isomorphism between one of the local systems in the pairing and the $\Q$-dual of the other.
 \end{remark}

\begin{remark}\label{rem:torsionRinfty}
Note that, since \(R_\infty\) is a flat \(R\)-module (see Definition~\ref{def:exp}), there are natural isomorphisms for every \(i\ge 0\):
\[
    R_\infty \otimes_R H^i(U, \oL) \cong 
  H^i(U,  R_\infty \otimes_R  \oL)
  \text{ and }
  R_\infty \otimes_R H_i(U, \cL) \cong
 H_i(U,  R_\infty \otimes_R \cL). 
\]
For any \(R\)-module \(A\), let \(A_1\) denote the submodule of elements that are annihilated by a power of \((t-1)\). Using the flatness again, \(\Tors_{R_\infty} (R_\infty \otimes_R A)\cong R_\infty \otimes_R \Tors_R A \cong A_1\). If we take the torsion part above, we have:
\[
    \Tors_R H^i(U, \oL)_1 \cong 
  \Tors_{R_\infty} H^i(U,  R_\infty \otimes_R  \oL)
  \text{ and }
  \Tors_R H_i(U, \cL)_1 \cong
 \Tors_{R_\infty} H_i(U,  R_\infty \otimes_R \cL). 
\]
\end{remark}

\begin{remark}\label{rem:quotientVStensor}
    For any \(A\)-module \(R\) and any \(m\ge 0\), the identity of \(A\) induces an isomorphism
    \[
    \frac{A}{(t-1)^mA} \cong R_m\otimes_R A.
    \]
    In particular, this can be applied to \(H^i(U, \oL)\) and  \(H_i(U, \cL)\) for any \(i\ge 0\), and to the sheaves \(\cL\) and \(\oL\) themselves. In \cite{EvaMoises}, the notation \(\oL_m\) is used to denote the quotient, and in \cite{abeliancovers}, the tensor products are used. Note that due to this isomorphism, they coincide canonically.
\end{remark}

 \subsection{Relation between homology and cohomology}

  We will need a way to relate homology and cohomology of local systems, while keeping track of the precise map, so we will state here the precise relation. This is well-known, and similar results can be found in \cite[Section 2.5]{dimca2004sheaves}.

\begin{proposition}\label{prop:chainHomology}
     Let \(S\) be a commutative ring, and let \(L\) be a local system of \(S\)-modules on a connected locally contractible space \(U\). Let \(x\in U\), let \(\pi_1 = \pi_1(U,x)\), and let \(\wt U\) be the universal cover of \(U\). Let \(C_\bullet(\wt U)\) be the singular chain complex of \(\wt U\) with coefficients in $S$ (placed such that $C_0(\wt U)$ is at degree $0$), which is a complex of free \(S[\pi_1]\)-modules. Let \(L_{x}\) denote the stalk of \(L\) at \(x\), seen as a complex of \(S[\pi_1]\)-modules. Consider the following complexes:
      \begin{equation}
      \begin{split}
      C_\bullet(U, L) &\coloneqq C_\bullet(\wt U) \otimes_{S[\pi_1]} L_x;\\
        C^\bullet(U, L) &\coloneqq \Hom_{S[\pi_1]}^\bullet(C_\bullet(\wt U) ,L_x).
      \end{split}   
      \end{equation}
Then, following \cite[IV.3]{whitehead} (see also \cite[Definition 2.10]{mhsalexander}), the \(i\)-th homology (i.e. cohomology in degree \(-i\)) of \(C_\bullet(U,L)\) is \(H_i(U,L)\), and \(H^i(C^\bullet(U,L) )\cong H^i(U,L)\).
\end{proposition}

\begin{definition}
    For a complex of local systems of \(S\)-modules \(L^\bullet\), we can make the analogous definition to the one above:
\[      \begin{split}      C_\bullet(U, L^\bullet) &\coloneqq \Tot^\bullet (C_\bullet(\wt U) \otimes_{S[\pi_1]} L_x^\bullet);\\
        C^\bullet(U, L^\bullet) &\coloneqq \Tot^\bullet(\Hom_{S[\pi_1]}^{\bullet,\bullet}(C_\bullet(\wt U) ,L_x^\bullet)).
      \end{split}   
\]
\end{definition}
\begin{theorem}\label{thm:dualHomology}
Let \(S\) be a commutative ring, let \(M\) be an \(S\) module, and let \(L_1, L_2\) be two \(S\)-local systems on a connected locally contractible space \(U\). Consider a locally defined bilinear pairing \(\langle \cdot , \cdot \rangle \colon L_1 \otimes_S L_2 \to \ul M\), or equivalently, a morphism \(L_1 \to \Homm_S(L_2, \ul M)\) Then,

\begin{enumerate}
    \item\label{part:pairing}The pairing \(\pairing\) induces a map of complexes
\[
C^\bullet(U,L_1)
\to
\Hom_S^\bullet(C_\bullet(U,L_2), M)
\]
\item\label{part:functorialPairing}Suppose we have two more local systems \(\wt L_i\), maps \(\psi_i\colon \wt L_i \to L_i \) for \(i=1,2\), an \(S\)-module \(\wt M\) and a map \(\psi_M\colon M\to \wt M\). Then, functoriality of the (co)chain complex and composition with \(\psi_M\) yields a map:
\[
C^\bullet(U, \wt L_1) \to
C^\bullet(U, L_1) \to \Hom_S^\bullet (C_\bullet(U,L_2 ), M)
\to
\Hom_S^\bullet (C_\bullet(U, \wt L_2), \wt M).
\]
This composition is induced by the pairing \(\psi_M\circ \langle \psi_1(\cdot), \psi_2(\cdot)\rangle\).
\item\label{part:naivePairingInHomology}
Taking cohomology yields composition maps \[
 H^i(U, L_1) \to H^i(\Hom_S^\bullet(C_\bullet(U,L_2), M) ) \to \Hom_S(H_i(U, L_2), M).
 \]
 \item\label{part:perfectPairing} If \(\pairing\) is perfect, i.e. it is an isomorphism \(L_1\cong \Homm_S(L_2, \ul M)\), then the map in \eqref{part:pairing} is an isomorphism of complexes. Furthermore, if \(M\) is an injective \(S\)-module, the maps in \eqref{part:naivePairingInHomology} are isomorphisms
 \(
  H^i(U, L_1) \cong \Hom_S(H_i(U, L_2), M)
 \)
 for all \(i\).
\end{enumerate}\end{theorem}
\begin{proof}

    Let \(x\), \(\pi_1\) and \(C_\bullet(\wt U)\) be as in Proposition~\ref{prop:chainHomology}. Let \(L_{i,x}\) denote the stalk of \(L_i\) at \(x\). 
    \begin{enumerate}
        \item 
    The induced map is as follows:
\[
      \begin{tikzcd}[row sep = 0.0em]
\Hom_{S[\pi_1]}(C_i(\wt U), L_{1,x})
\arrow[r]
& 
\Hom_S\left(C_i(\wt U) \otimes_{S[\pi_1]} L_{2,x},M\right)
\\
\phi
\arrow[r, mapsto]
&
\left(\gamma\otimes b \mapsto \langle \phi(\gamma),b\rangle\right).
    \end{tikzcd}
    \]
   \item  Part \eqref{part:functorialPairing} is straightforward from the definition.
    
\item For part \eqref{part:naivePairingInHomology}, one can simply write the definition of \(H^i(\Hom_S^\bullet(C_\bullet(U,L_2), M) )  \) and notice that there is a naively defined map from \(
 H^i(\Hom_S^\bullet(C_\bullet(U,L_2), M) )
\), which is a subquotient of \(\Hom_S(C_i(U, L_2), M)\), to \(\Hom_S(H^i(C_\bullet(U,L_2)), M)\), given by restricting a homomorphism from \(C_i(U, L_2)\) to \(H_1(C_\bullet(U, L_2))\), its subquotient.
    
   \item Let us denote \(L_2=L\) and \(L_1 = \Homm(L, M)\). For part \eqref{part:perfectPairing}, the map above is the tensor-hom adjunction isomorphism (here we are using that for a local system, \(\Homm\) commutes with taking stalks):
   \begin{multline*}
   C^\bullet(U, \Homm_S(L, \ul M)) = 
   \Hom_{S[\pi_1]}^\bullet(C_\bullet(\wt U), \Homm_S(L_x, M))\cong 
   \\
   \Hom_S^\bullet(C_\bullet(\wt U) \otimes_{S[\pi_1]} L_x, M) = 
   \Hom_S^\bullet(C_\bullet(U ,L), M).
   \end{multline*}
   Taking cohomology on the left hand side, we obtain \(H^i(U, \Homm_S(L, \ul M))\). Taking cohomology on the right hand side of the isomorphism, and using the hypothesis that \(M\) is injective, we obtain:
   \[
      H^i(\Hom_S^\bullet(C_\bullet(U ,L), M)) \cong
      \Hom_S(H^{-i}(C_\bullet(U,L)), M)
      \cong 
      \Hom_S(H_{i}(U,L), M).
   \]
\end{enumerate}

\end{proof}

We will need the following version of the Universal Coefficient Theorem, which can be found in \cite[Lemma 2.12]{mhsalexander} with a very similar formulation. We include its proof because the precise definition of the morphism $\mathrm{UCT}$ will be important in the proof of Lemma~\ref{lem:commutes}.

\begin{theorem}[Universal Coefficient Theorem]\label{thm:UCT}
    Let \(U\) be a connected locally contractible space, and let \(L\) be a local system of free \(R\)-modules on \(U\). For every \(i\), there is an injective map that is an isomorphism onto the torsion:
    \begin{equation}\label{eq:mapUCT}
    \mathrm{UCT}:\Ext^1_R(\Tors_R(H_i(U, L)), R)
    \hookrightarrow
    H^{i+1}(U, \Homm_R(L, R)).
    \end{equation}
    This map is defined as follows: Let \(K\) be the field of fractions of \(R\). Then, \(I^\bullet \coloneqq ( K\to K/R)\) is an injective resolution of \(R\) (with the inclusion \(R\to K\)), and there is a map in the derived category \(K/R\to I^\bullet[1] \xrightarrow{\cong} R[1]\). Then, the map \(\mathrm{UCT}\) in~\eqref{eq:mapUCT} is the composition of the dashed arrows below. Note that the dashed arrows are uniquely determined by this diagram.
    \[
    \begin{tikzcd}[row sep = 1.1em]
    \Ext^1_R(H_i(U, L), R) \arrow[d, "\cong"] 
    &
    \Hom_R(H_i(U, L), K/R)
    \arrow[d, twoheadrightarrow]
    \arrow[l, "K/R \to R{[1]}"', twoheadrightarrow]
    \arrow[r, "K/R \to R{[1]}"]
    &
    H^i(\Hom_R^\bullet(C_\bullet(U, L), R[1])) 
    \arrow[d, "\text{Theorem \ref{thm:dualHomology}\eqref{part:perfectPairing}}"', "\cong"]
    \\
    \Ext^1_R(\Tors_R(H_i(U, L)), R)
    \arrow[r, dashrightarrow, "\cong"]
    &
    \Hom_R(\Tors_R H_i(U, L), K/R)
    \arrow[r, dashrightarrow]
    &
    H^{i+1}(U,\Homm_R(L, R)).
    \end{tikzcd}
    \]
\end{theorem}
\begin{proof}
    Considering \(C_\bullet(U, L)\), as in Proposition \ref{prop:chainHomology}, we consider the short exact sequence \(0\to K/R \to I^\bullet[1] \to K[1]\to 0\), and we apply \(\Hom^\bullet_R(C_\bullet(U, L), \cdot)\) (recall that \(\Hom^\bullet \colon \Tot \circ \Hom^{\bullet,\bullet}\)), to obtain a short exact (since \(C_\bullet(U, L)\) is a complex of free \(R\)-modules) sequence:
\[
0
\to \Hom^\bullet_R(C_\bullet(U, L), K/R )
\to \Hom^\bullet_R(C_\bullet(U, L), I^\bullet[1]) 
\to \Hom^\bullet_R(C_\bullet(U, L), K[1])
\to 0.
\]
If we let \(I = K \) or \(I=K/R\), then \(I\) is injective, and taking the cohomology of these complexes we obtain:
\begin{align*}
H^i(
\Hom^\bullet_{R}(C_\bullet(U,L), I)
)
&\cong
\Hom_{R}(H^{-i}(C_\bullet(U,L)), I)
&
(\text{\(I\) is injective})
\\
&=
\Hom_{R}(H_{i}(U,L), I).
&
\text{(Theorem \eqref{thm:dualHomology}~\eqref{part:perfectPairing})}
\end{align*}
Also, we have a quasi-isomorphism \(I^\bullet \cong R\) given by \(R\to K\), so up to this map, we have,
\begin{align*}
   H^i(\Hom^\bullet_{R}(C_\bullet(U,L),I^\bullet[1])) &\cong
    H^i(R\Hom^\bullet_{R}(C_\bullet(U,L),R[1])) 
    &\text{(\(I^\bullet \cong R\))}
    \\
    &\cong
    H^i(\Hom^\bullet_{R}(C_\bullet(U,L),R[1])) 
    & \text{(\(C_\bullet(U,L)\) is free)}\\
    &\cong
    H^{i+1}(C^\bullet(U,\Homm_R(L, R))) 
    & \text{(Theorem \ref{thm:dualHomology}{,} \eqref{part:perfectPairing})}\\
    &\cong
    H^{i+1}(U,\Homm_R(L, R)) .
    & \text{(Proposition \ref{prop:chainHomology})}
\end{align*}
Therefore, the short exact sequence above induces the long exact sequence
\begin{equation}\label{eq:UCTles}
\cdots \to\Hom_R(H_i(U,L), K)
\to
\Hom_R(H_i(U,L), K/R)
\to
H^{i+1}(U, \Homm_R(L, R))
\to \cdots
\end{equation}
Note that, by the description of the maps above, the second map is indeed induced by \(K/R\to R[1]\), since it was constructed as the composition of the inclusion of \(K/R\) into \(I^\bullet[1]\) with the inverse of the quasi-isomorphism \(R[1]\to I^\bullet[1] \) and the isomorphism from Theorem \ref{thm:dualHomology}.

Next, since \(K\to K/R\) is an injective resolution of \(R\), for any \(R\)-module \(A\),
\(
\coker(
\Hom_R(A, K)
\to
\Hom_R(A, K/R))\cong \Ext^1_R(A, R)
\), and furthermore this map is induced by \(I^\bullet\cong R\), so we have
\[
0\to \Ext^1_R(H_i(U,L), R)
\to
H^{i+1}(U, \Homm_R(L, R))\to \cdots 
\]
In particular, the composition
\[
\Hom_R(H_i(U,L), K/R) \to 
\coker(
\Hom_R(H_i(U,L), K)
\to
\Hom_R(H_i(U,L), K/R))
\to
\Ext^1_R(H_i(U,L), R)
\]
is indeed induced by \(K/R\to R[1]\) as well. Note that only the torsion part contributes to the \(\Ext\) group, so the above map factors through the quotient \(\Ext^1_R(\Tors_R H_i(U,L), R)\):
\[
0\to \Ext^1_R(\Tors_R H_i(U,L), R)
\to
H^{i+1}(U, \Homm_R(L, R))\to \cdots 
\]
It only remains to prove that the map is an isomorphism onto the torsion: it is injective due to the exactness of \eqref{eq:UCTles}. By the exactness, its cokernel is isomorphic to the degree zero cohomology of \(\Hom_R(H_{i+1}(U, L), I^\bullet)\). Since \(I^\bullet\) is an injective resolution of \(R\), the cokernel is isomorphic to \(\Hom_R(H_{i+1}(U, L), R)\),
which is a free module. This shows that indeed the map is surjective onto the torsion.
\end{proof}

\begin{corollary}\label{cor:UCTinfty}
    Let \(K_\infty\) be the field of fractions of \(R_\infty\). Since \(R_\infty\) is flat over \(R\), everything in Theorem~\ref{thm:UCT} can be tensored with \(R_\infty\) to obtain an analogous result. In particular, we obtain the following commutative diagram:
     \[
    \begin{tikzcd}[row sep = 1.1em, column sep  =4em]
    \Hom_{R_\infty}(H_i(U, R_\infty \otimes_R L), {K_\infty}/{R_\infty})
    \arrow[ddr, twoheadrightarrow]
    \arrow[d, twoheadrightarrow, , "K_\infty / R_\infty \to R_\infty {[1]}"']
    \arrow[r, "K_\infty / R_\infty \to R_\infty {[1]}"]
    &
    H^i(\Hom_{R_\infty}^\bullet(C_\bullet(U, R_\infty \otimes_R L), {R_\infty}[1])) 
    \arrow[d, "\text{Theorem \ref{thm:dualHomology}\eqref{part:perfectPairing}}"', "\cong"]
    \\
    \Ext^1_{R_\infty}(H_i(U,  R_\infty \otimes_R L), {R_\infty}) \arrow[d, "\cong"] 
    &
    H^{i+1}(U,\Homm_{R_\infty}(R_\infty \otimes_R L, {R_\infty}))
    \\
    \Ext^1_{R_\infty}(\Tors_{R_\infty}(H_i(U, R_\infty \otimes_R L)), {R_\infty})
    \arrow[r, dashrightarrow, "\cong"]
    &
    \Hom_{R_\infty}(\Tors_{R_\infty} H_i(U, R_\infty \otimes_R L), {K_\infty}/{R_\infty}).
    \arrow[u, dashrightarrow]
    \end{tikzcd}
    \]
\end{corollary}

Finally, let us recall a result that can be proved using Proposition~\ref{prop:chainHomology}, and will be of use later.

\begin{proposition}\label{prop:TorsIntoHomology}
Let \(L\) be a local system of finitely generated free \(R\)-modules (we will apply this statement for \(\cL\) and \(\oL\)). The projection $L\twoheadrightarrow R_m\otimes_RL$ induces monomorphisms
\begin{equation}\label{eq:commutetensorhomology}
\frac{H_i(U,L)}{(t-1)^mH_i(U,L)} \hookrightarrow H_i(U, R_{m}\otimes_R L)\text{ and }
\frac{H^i(U,L)}{(t-1)^mH^i(U,L)} \hookrightarrow H^i(U, R_{m}\otimes_R L).
\end{equation}
Furthermore, if \(m\gg 1\), there are also monomorphisms induced by the identity of \(H_i(U, L)\) and \(H^i(U, L)\), respectively:
\begin{equation}\label{eq:commutetensorhomologyTors}
\Tors_R H_i(U, L)_1 \hookrightarrow 
\frac{H_i(U,L)}{(t-1)^mH_i(U,L)}\text{ and }
\Tors_R H^i(U, L)_1 \hookrightarrow \frac{H^i(U,L)}{(t-1)^mH^i(U,L)} .
\end{equation}
The subindex $1$ denotes the submodule annihilated by a power of \(t-1\), as in Remark~\ref{rem:torsionRinfty}.
\end{proposition}
\begin{proof}
    We apply Proposition~\ref{prop:chainHomology}, and we have that the projection \(L\twoheadrightarrow R_m\otimes_R L\) induces the following morphisms:
    \begin{align*}
            R_m\otimes_R C_\bullet(U, L)
        =
        R_m\otimes_R (C_\bullet(\wt U)\otimes_{R[\pi_1]} L_x)
        &\cong
        C_\bullet(\wt U)\otimes_{R[\pi_1]} (R_m\otimes_R L_x)
        =
        C_\bullet(U,R_m\otimes_R  L)  
        ;\\
        R_m\otimes_R C^\bullet(U, L)
        =
        R_m\otimes_R \Hom^\bullet_{S[\pi_1]}(C_\bullet(\wt U), L_x)
        &\cong
        \Hom^\bullet_{S[\pi_1]}(C_\bullet(\wt U), R_m\otimes_R L_x)
        =
        C^\bullet(U, R_m\otimes_R L) 
        .
    \end{align*}
    Note that indeed both morphisms are induced by the map on stalks \(1\otimes_R \Id_{L_x} \colon L_x \to R_m\otimes_R L_x\).
    
    If we let \(C\) denote either \(C_\bullet(U, L)\) or \(C^\bullet(U,L)\), it is a complex of finitely generated free \(R\)-modules. The question is whether \(\frac{H^i(C)}{(t-1)^mH^i(C)}\to H^i(R_m\otimes_R C)\) is an injection. This follows from the long exact sequence in cohomology applied to the short exact (because \(C\) is a complex of free modules) sequence
    \[
    0\to C\xrightarrow{(t-1)^m} C \to R_m\otimes_R C\to 0.
    \]
    It induces the short exact sequence
    \[
    0\to \frac{H^i(C)}{(t-1)^mH^i(C)} \to H^i(R_m\otimes_R C),
    \]
    so indeed the maps~\eqref{eq:commutetensorhomology} are injective.
    When \(m\) is large enough that \((t-1)^m\cdot \Tors_R H^i(C)_1 = 0\) (using the fact that \(L\) is finitely generated), then \(\Tors_R H^i(C)_1\) embeds into \(\frac{H^i(C)}{(t-1)^mH^i(C)} \), so indeed the maps~\eqref{eq:commutetensorhomologyTors} are injective.
\end{proof}

  \subsection{Mixed Hodge complexes}

Our constructions of MHS rely on Deligne's machinery of mixed Hodge complexes of sheaves from \cite{DeligneII}. However, we will not need to work with them explicitly here, but instead rely on previous work (see Section~\ref{sec:previousMHS}). Therefore, we will not give the definition, which can be found in \cite[Definition 3.13]{peters2008mixed}. The mixed Hodge complexes of sheaves we use are over $\Q$.

  \begin{definition}[{\cite[Lemma-Definition 2.35]{peters2008mixed}}]\label{def:shift}
  Let \(\cF^\bullet\) be a mixed Hodge complex of sheaves, and let \(m\in \Z\). There is a notion of a Tate twist, \(\cF^\bullet(m)\), extending the Tate twist of pure and mixed Hodge structures, which does not alter the underlying complex (although we choose to not multiply by $(2\pi i)^m$ as in \cite{peters2008mixed} and instead follow the convention explained in \cite[Section 2.8]{mhsalexander}). There is also a notion of shift \(\cF^\bullet[m]\), that extends the cohomological shift of complexes.
\end{definition}

\begin{remark}
    The book \cite{peters2008mixed} contains two different definitions of shift of a mixed Hodge complex of sheaves: one is found in Lemma-Definition 2.35, and the second one is used implicitly in Theorem 3.22. We use the former. For the precise definitions and the difference between these, see \cite[Remark 2.33]{mhsalexander}. In this paper, we will just need the properties of the shift and Tate twist stated in Theorem~\ref{thm:MHSDeligne} below.
\end{remark}

\begin{defprop}[The cone of a morphism]\label{def:mappingCone}
      Let \(\phi\colon \cK^\bullet_1\to \cK_2^\bullet\) be a morphism of complexes in an abelian category (e.g. a category of sheaves of modules over a commutative ring). We can define a new complex \(\Cone^\bullet(\phi)\). We consider a double complex \(\cK^{\bullet, \bullet}\), where,  \(\cK^{-1, \bullet} = \cK_{1}^\bullet\), \(\cK^{0, \bullet} = \cK_{2}^\bullet\) and the remaining rows vanish. The horizontal differentials are those in the original complexes, and the vertical differential is \(\phi\). Then, we define the cone as \(
       \Cone^\bullet(\phi) \coloneqq \Tot^\bullet (\cK^{\bullet,\bullet})
      \). 
      In particular, as a vector space, \(\Cone^\bullet(\phi) \cong \cK^\bullet_2 \oplus \cK^\bullet_1[1]\). The stupid filtration on the rows of \(\cK^{\bullet,\bullet}\) yields a short exact sequence:
  \[
    0\to \cK_2^\bullet \to \Cone^\bullet (\phi) \to \cK_1^\bullet[1] \to 0.
  \]
\end{defprop}

\begin{theorem}[{\cite[Theorem 3.22]{peters2008mixed}}]\label{thm:mixedCone}
  Let \(\phi\colon \cK^\bullet_1\to \cK_2^\bullet\) be a morphism of mixed Hodge complexes of sheaves. Then \(\Cone^\bullet(\phi)\) has a canonical structure of a mixed Hodge complex of sheaves. The mapping cone short exact sequence becomes a short exact sequence of mixed Hodge complexes of sheaves:
  \[
    0\to \cK_2^\bullet \to \Cone^\bullet (\phi) \to \cK_1^\bullet[1](-1) \to 0.
  \]
  Note that the presence of the Tate twist does not seem to agree with \cite{peters2008mixed}, but the definition of shift implicitly used therein only matches our Definition~\ref{def:shift} up to a Tate twist.
\end{theorem}

The main property that mixed Hodge complexes satisfy is that the hypercohomology of their rational part carries a canonical MHS, which behaves as follows after taking Tate twists or shifts.

\begin{theorem}[{\cite[Theorem 3.18]{peters2008mixed}, \cite[Remark 2.33]{mhsalexander}}]\label{thm:MHSDeligne}
    If \(K^\bullet\) is the rational part of a \(\Q\)-mixed Hodge complex of sheaves on a topological space \(X\), then for every \(i\ge 0\), \(\bH^i(X, K^\bullet)\) carries a MHS. Furthermore, \(\bH^i(X, K^\bullet(1)) \cong \bH^i(X, K^\bullet)(1)\), and \(\bH^i(X, K^\bullet[1]) \cong \bH^{i+1}(X, K^\bullet)(1)\).
\end{theorem}

\section{Summary of previous work}\label{sec:previousMHS}

\subsection{Thickened Mixed Hodge complexes}

  We will use the thickened mixed Hodge complex of sheaves from \cite{mhsalexander,abeliancovers}. We will not make use of its definition, which can be found in \cite[Definition 10.5]{abeliancovers}, so we will not recall it here. We will simply summarize the properties we are concerned with.

  \begin{theorem}\label{thm:MHC}
    Let \(U,f\) be as in Definition \ref{def:Uf}, let \(\cL,\oL\) be as in Definitions \ref{def:cL} and \ref{def:oL}, respectively. Let \(j\colon U\to X\) be a smooth compactification such that \(X\setminus U\) is a divisor with simple normal crossings. For every \(m\in \Z\setminus\{0\}\), there is a \(\Q\)-mixed Hodge complex of sheaves, whose rational part we will denote \(\sK^\bullet_m\), with the following properties.
  \begin{enumerate}
    \item\label{part:toX} The adjunction map induces a quasi-isomorphism \(\sK_m^\bullet \to R j_* j^{-1} \sK_m^\bullet\). 
    \item\label{part:nu} \(\sK_m^\bullet\) is a complex of sheaves of modules over $R_\infty$, and there is an \(R_\infty\)-linear quasi-isomorphism \(\nu_m\colon R_m\otimes_R \oL \to {j^{-1}}\sK^\bullet_m\). 
    \item\label{part:KmfromK1} As vector spaces (ignoring the differential), \(\sK^\bullet_m \cong R_m \otimes_\Q \sK^\bullet_1\).
    \item\label{part:inclusionProjection} For \(m_1\ge m_2>0\), the projection \(R_{m_1}\twoheadrightarrow R_{m_2}\) and its dual \(R_{-m_2}\hookrightarrow R_{-m_1}\) induce morphisms \(\sK^\bullet_{m_1} \twoheadrightarrow \sK^\bullet_{m_2}\) and \(
\sK^\bullet_{-m_2}\hookrightarrow \sK^\bullet_{-m_1}\), respectively. These come from morphisms of  mixed Hodge complexes of sheaves.

These agree, via \(\nu_{m_1},\nu_{m_2}\) and after applying \(j^{-1}\), with the corresponding morphisms\linebreak \(R_{m_1}\otimes_R \oL \twoheadrightarrow R_{m_2}\otimes_R \oL\) and \(R_{-m_2}\otimes_R \oL\hookrightarrow R_{-m_1}\otimes_R \oL\).
    \item\label{part:multiplication} Multiplication by \(s\), as a map \(R_m\to R_m\), induces a morphism \(\sK^\bullet_m(1) \to \sK^\bullet_m\), which is the rational part of a morphism of mixed Hodge complexes of sheaves. Here, \((1)\) denotes a Tate twist, i.e. \(\sK^\bullet(1)\) is the rational part of the Tate twisted mixed Hodge complex of sheaves. Via \(\nu_m\) this map corresponds to multiplication on \(R_m\otimes_R \oL\).
    \item\label{part:dual}If \(m> 0\), there are isomorphisms \(\bD_m\colon \sK^\bullet_m(1-m) \cong \sK^\bullet_{-m}\), induced by the isomorphism \(R_m\cong R_{-m}\) in Remark~\ref{rem:RmDualMultiplication}. These come from morphisms of mixed Hodge complexes of sheaves (again, the Tate twist \((1-m)\) is applied to the mixed Hodge complex). They agree with the corresponding isomorphisms \(R_m\otimes_R \oL \cong R_{-m} \otimes_R \oL\) via \(\nu_{\pm m}\).
  \end{enumerate}
\end{theorem}
\begin{proof}
  The complex is defined in~\cite[Definition 10.5]{abeliancovers}. As for its properties:
  \begin{enumerate}
    \item This can be found in \cite[Remark 10.10]{abeliancovers}, which repeats the construction of \cite[Definition 6.1]{abeliancovers}.
    \item The quasi-isomorphism is defined in~\cite[Construction 10.8]{abeliancovers} and proved to be a quasi-\linebreak isomorphism in~\cite[Remark 10.10]{abeliancovers}.
    \item This is part of the definition.
    \item This is \cite[Proposition 3.13]{abeliancovers}. The fact that it agrees with \(\nu\) is a direct consequence of its definition.
    \item By~\cite[Proposition 3.12]{abeliancovers}, multiplication induces a morphism \(H_1(\C^*, \Q) \otimes_\Q \sK^\bullet_m\to \sK^\bullet_m\). Note that \(H_1(\C^*, \Q)\) is one-dimensional and pure of type \((-1,-1)\), so its tensor is a Tate twist by \((1)\). It agrees with \(\nu_m\) because \(\nu_m\) is \(R_\infty\) linear.
    \item This is~\cite[Example 4.13]{abeliancovers}.

  \end{enumerate}
\end{proof}

\subsection{The MHS on the quotients of Alexander modules}

\begin{remark}
In \cite[Remark 10.12]{abeliancovers}, it is shown that the definitions of \(\sK_m^\bullet\) and \(\nu_m\) agree with the analogous notions corresponding to the mixed Hodge complex of sheaves defined in \cite[Theorem 5.24]{mhsalexander}.
\end{remark}

The purpose of this paper is to show that the constructions in \cite{abeliancovers} and \cite{mhsalexander} produce the same MHS. In Definitions \ref{def:endowedMHS}, \ref{def:endowedMHSHomology} and \ref{def:proMHS} we will outline the construction of \cite{abeliancovers}.

\begin{definition}\label{def:endowedMHS}
    With the notations above, for every \(i\ge 0\) and \(m \in \Z\setminus\{0\}\), \(H^i(U, R_{m}\otimes_R \oL)\) carries a mixed Hodge structure defined in \cite[Definition 6.1]{abeliancovers}.
    \begin{itemize}
        \item If $m<0$, it is defined by giving \(\bH^i(X, \sK^\bullet_{m})\) a MHS via Theorem \ref{thm:MHSDeligne}, and translating it via the following isomorphisms:
    \begin{align*}
        H^i(U, R_{m}\otimes_R \oL)
        \overset{\nu_{m}}\cong
        \bH^i(U, j^{-1}\sK^\bullet_{m})
        \cong
        \bH^i(X, Rj_*j^{-1}\sK^\bullet_{m})
        \overset{\text{Theorem \ref{thm:MHC}} \eqref{part:toX}}\cong
        \bH^i(X, \sK^\bullet_{m}).
    \end{align*}
    \item If $m>0$, it is defined similarly but with a shift, namely through the following isomorphisms:
    \begin{align*}
        H^i(U, R_{m}\otimes_R \oL)
        \overset{\nu_{m}}\cong
        \bH^{i-1}(U, j^{-1}\sK^\bullet_{m}[1])
        \cong
        \bH^{i-1}(X, Rj_*j^{-1}\sK^\bullet_{m}[1])
        \overset{\text{Theorem \ref{thm:MHC}} \eqref{part:toX}}\cong
        \bH^i(X, \sK^\bullet_{m})(1).
    \end{align*}
      \end{itemize}
\end{definition}

\begin{remark}\label{rem:MHSmaps}
    Applying Theorem \ref{thm:MHC}, the following maps are MHS morphisms, for any \(m'\ge m \geq 1 \):
    \begin{itemize}
        \item The map \(H^i(U, R_{m'}\otimes_R \cL)\to H^i(U, R_{m}\otimes_R \cL)\) induced by the projection \(R_{m'} \to R_{m}\).
        \item The map \(H^i(U, R_{-m}\otimes_R \cL)\to H^i(U, R_{-m'}\otimes_R \cL)\) induced by the inclusion \(R_{-m} \to R_{-m'}\).
        \item Multiplication by \(s\), as a map \(H^i(U, R_{m}\otimes_R \cL)(1)\to H^i(U, R_{m}\otimes_R \cL)\). Note that by Theorem \ref{thm:MHSDeligne}, Tate twists on a mixed Hodge complex induce Tate twists on its hypercohomology.
    \end{itemize}
\end{remark}

 \begin{definition}[MHS on \(H_i(U, R_{m}\otimes_R \cL)\)]\label{def:endowedMHSHomology}
    With the notations above, for every \(i\ge 0\) and \(m\in \Z\setminus\{0\}\), \(H_i(U, R_{m}\otimes_R \cL)\) carries a mixed Hodge structure defined in \cite[Definition 6.6]{abeliancovers}. It is defined as follows. Combining Theorem \ref{thm:dualHomology}\eqref{part:perfectPairing} and Remark \ref{rem:RmcLDual}, we have an isomorphism
    \[
    H_i(U, R_m\otimes_R \cL) \xrightarrow[\cong]{\Theta} \Hom_\Q(H^i(U, R_{-m}\otimes_R \oL) , \Q).
    \]
    The codomain of $\Theta$ is endowed with the dual MHS of the one in Definition \ref{def:endowedMHS}. We define the MHS on the domain of $\Theta$ as the one induced by the MHS on $\Hom_\Q(H^i(U, R_{-m}\otimes_R \oL) , \Q)$ through the isomorphism $\Theta$. 

    Endowing the $\Q$-dual spaces with the corresponding $\Q$-dual MHS, the dual map $\Theta^\vee$ is a MHS morphism.
 \end{definition}

 \begin{remark}
     The dual maps of those in Remark \ref{rem:MHSmaps} are MHS morphisms. For every \(i\ge 0\) and \(m'\ge m\geq 1\):
     \begin{itemize}
         \item The map \(H_i(U,R_{m'} \otimes_R \cL)\to H_i(U,R_{m} \otimes_R \cL)\) induced by the projection \(R_{m'} \twoheadrightarrow R_m\) (the dual of the map in cohomology induced by the inclusion \(R_{-m}\hookrightarrow R_{-m'}\)).
         \item The map \(H_i(U,R_{-m} \otimes_R \cL)\to H_i(U,R_{-m'} \otimes_R \cL)\) induced by the inclusion \(R_{-m} \hookrightarrow R_{-m'}\) (the dual of the map in cohomology induced by the projection \(R_{m'}\twoheadrightarrow R_{m}\)).
         \item Multiplication by \(s\), as a map \(H_i(U, R_m\otimes_R \cL)(1) \to H_i(U, R_m\otimes_R \cL)\).
     \end{itemize}
 \end{remark}

 \begin{definition}\label{def:proMHS}
With the notations above, for every \(i\ge 0\), both \(H^i(U, R_\infty \otimes_R , \oL)\cong R_\infty \otimes_R H^i(U, \oL) \) and \(H_i(U, R_\infty\otimes_R \cL)\cong R_\infty \otimes_R H_i(U,\cL)\) carry a pro-MHS, i.e. each is an inverse limit of mixed Hodge structures. They are defined in \cite[Remark 6.5 and Definition 6.6]{abeliancovers}, in the following way: For every \(i\ge 0\) and \(m'\ge m>0\), the projection maps induce MHS morphisms \(H^i(U, R_{m'} \otimes_R \oL) \to H^i(U, R_m\otimes_R \oL)\) and \(H_i(U,R_{m'} \otimes_R \cL)\to H_i(U,R_{m} \otimes_R \cL)\). Therefore, the limits \( \varprojlim_m H^i(U,R_{m} \otimes_R \oL)\) and \(\varprojlim_m H_i(U,R_{m} \otimes_R \cL)\) each have pro-MHS. In \cite[Corollary 2.29]{abeliancovers}, it is shown that the natural maps
\begin{align*}
R_\infty \otimes_R H^i(U, \oL) &\to \varprojlim_m H^i(U, R_m\otimes_R \oL)&&\text{ and}
&R_\infty \otimes_R H_i(U, \cL) &\to \varprojlim_m H_i(U, R_m\otimes_R \cL)
\end{align*}
are isomorphisms, so the domain of each map  acquires a MHS as well.
\end{definition}

\begin{defprop}\label{defprop:MHSquotients}
  From Proposition~\ref{prop:UfvscL}, \(H_i(U, \cL)\cong H_i(U^f, \Q)\), so Definition~\ref{def:proMHS} also defines a MHS on \(R_\infty \otimes_R H_i(U^f, \Q)\) for every \(i\). From \cite[Corollary 6.16]{abeliancovers}, every vector space quotient of the form \(R_m\otimes_R  H_i(U^f, \Q)\) for \(m\ge 1\) has a (necessarily unique) MHS that makes the quotient map a MHS morphism. In particular, the quotient maps \(R_{m'}\otimes_R  H_i(U^f, \Q)\twoheadrightarrow R_m\otimes_R  H_i(U^f, \Q)\) for \(m'\ge m\geq 1\) are also MHS morphisms.
\end{defprop}

\begin{defprop}[{\cite[Propositions 8.3, 8.4]{abeliancovers}}]\label{defprop:MHSquotientsN}
    For any \(m, N\ge 1\), there is a canonical MHS on \(\frac{H_i(U^f, \Q)}{(t^N-1)^m H_i(U^f, \Q)}\) that agrees with the one on Definition-Proposition~\ref{defprop:MHSquotients} for \(N=1\). For any \(N'\) which is a multiple of \(N\) and any \(m'\ge m\), the quotients maps are MHS morphisms:
    \[
    \frac{H_i(U^f, \Q)}{(t^{N'}-1)^m H_i(U^f, \Q)}
    \twoheadrightarrow
    \frac{H_i(U^f, \Q)}{(t^N-1)^m H_i(U^f, \Q)}
    ;\quad
    \frac{H_i(U^f, \Q)}{(t^N-1)^{m'} H_i(U^f, \Q)}    \twoheadrightarrow
    \frac{H_i(U^f, \Q)}{(t^N-1)^m H_i(U^f, \Q)}.
    \]
\end{defprop}

\subsection{The MHS on the torsion part of Alexander modules}

In the remainder of this section we will outline the construction of the MHS in \cite{mhsalexander}. For a survey of the construction and main properties, see \cite{mhsSurvey}.

\begin{definition}\label{def:mhsalexander}
  For any \(i\ge 0\), there is a canonical MHS on \(\Tors_R H_i(U^f, \Q)\), that can be found in \cite[Corollary 5.26, Remark 5.27]{mhsalexander}. An outline of the construction follows.
\end{definition}

\begin{construction}[Construction of the MHS on the torsion part]\label{con:mhsalexander}

  By \cite[Proposition 2.24]{mhsalexander} (based on \cite[Proposition 4.1]{BudurLiuWang}), there exists some \(N\in \N\) such that \((t^N-1)^m\) annihilates the \(\Tors_R H_i(U^f, \Q)\) for all $i$ and for sufficiently large $m$. The index \(N\) subgroup \(\langle t^N\rangle \subset \pi_1(\C^*)\) induces a degree \(N\) cover of \(U\) that we denote \(U_N\). It has a map \(f_N\colon U_N\to \C^*\) (an \(N\)-th root of \(f\)), which induces an infinite cyclic cover \((U_N)^{f_N}\to U_N\). The generator of its deck transformation group \(\pi_1(\C^*)\) acts on \(H_i\left((U_N)^{f_N}, \Q\right)\) as \(t^N\), giving it the structure of a module over $\Q[t^{\pm N}]$. In \cite[Lemma 2.26]{mhsalexander}, it is proven that \(H_i(U^f, \Q) \cong H_i((U_N)^{f_N}, \Q) \) as \(\Q[t^{\pm N}]\)-modules, and this isomorphism is induced by an isomorphism $\theta_N:(U_N)^{f_N}\to U^f$ of complex analytic varieties. Therefore, the action of \(\pi_1(\C^*)\) by deck transformations on the Alexander module of \((U_N, f_N)\) is  unipotent. Moreover, $\theta_N$ induces an isomorphism
  $$
  \theta_N\colon\Tors_{\Q[t^{\pm N}]} H_i\left((U_N)^{f_N},\Q\right)_1\xrightarrow{\cong} \Tors_R H_i(U^f,\Q),
  $$
  where the subindex $1$ denotes the generalized eigenspace of eigenvalue $1$ for a generator of the deck group of the cover $(U_N)^{f_N}\to U_N$. Hence, it suffices to endow the domain of this isomorphism with a MHS, and use it to endow the codomain with a MHS. In other words, it suffices to consider the case where the $t$-action on $H_i(U^f,\Q)$ is unipotent, and endow  $\Tors_R H_i(U^f,\Q)_1=\Tors_R H_i(U^f,\Q)$ with a canonical MHS.

  The construction begins with the MHS on \(H^i(U, R_m\otimes_R \oL)\) from Definition~\ref{def:endowedMHS} in the case when $m\geq 1$ (see \cite[Remark 10.0]{abeliancovers}, which explains that the constructions of the MHS on \(H^i(U, R_m\otimes_R \oL)\) are the same in both \cite{mhsalexander} and \cite{abeliancovers}). For \(m \gg 1\), it is shown in \cite[Corollary 3.9]{mhsalexander} that the natural map  \(\Tors_R H^i(U, \oL)\to H^i(U, R_m\otimes_R \oL)\) is an embedding (provided the torsion is unipotent). Furthermore, in \cite[Corollary 5.19]{mhsalexander}, it is shown that its image is a sub-MHS, so \(\Tors_R H^i(U,\oL)\) acquires a MHS.

  Finally, one shows that
  \begin{equation}\label{eq:simpleTorsionDuality}
    \Tors_R H_i(U, \cL) \cong \Hom_\Q(\Tors_R H^{i+1} (U, \oL),\Q),
  \end{equation}
  so the left hand side has a MHS, dual to the one on \(\Tors_R (U, \oL)\).
\end{construction}

We will need to understand the last isomorphism, so we will write the definition in more detail.

\begin{proposition}[Duality of the torsion Alexander modules]\label{prop:ResUCT}
  The map \eqref{eq:simpleTorsionDuality} is defined from the following diagram, taking duals:
\begin{equation}\label{eq:torsionSide-general}
    \Hom_\Q(\TorsH, \Q) \xleftarrow[\cong]{\Res}
\Ext^1_R(\TorsH, R)
\xrightarrow[\cong]{\mathrm{UCT}}
    \Tors_R H^{i+1}(U, \oL).
  \end{equation}
  Here, ``UCT'' denotes the morphism from Theorem \ref{thm:UCT}. To define Res, we use the same injective resolution \(R \to K\to K/R\) as in said theorem, and we take the following composition:
  \[
    \Hom_\Q(\TorsH, \Q) \xleftarrow{\res_*}
    \Hom_R(\TorsH, K/R) \xrightarrow{K/R \to R[1]}
\Ext^1_R(\TorsH, R).
  \]
  Here, \(\res_*\) is the composition with \(\res\colon K\to \Q\) from Remark~\ref{rem:comparisonResidues} (or rather, its induced map $\res:K/R\to\Q$), and the second arrow is given by postcomposition with the map in the derived category (as in Theorem \ref{thm:UCT}).
\end{proposition}

Construction \ref{con:mhsalexander} is fairly involved. However, using the work in \cite{EvaMoises}, we can skip the first step and compute part of the MHS without passing to a degree \(N\) cover.

\begin{theorem}[{\cite[Corollary 1.2]{EvaMoises}}]\label{thm:eigenspaces}
  Let \(\lambda\in \ov\Q\), and let \(g_\lambda(x)\in \Q[x]\) be its minimal polynomial, i.e. the monic irreducible polynomial that satisfies \(g_\lambda(\lambda)=0\). For any \(\R\)-module \(A\), we can consider its generalized eigenspace with eigenvalue \(\lambda\) or, more precisely, the sum of generalized eigenspaces with eigenvalues that are Galois conjugate to \(\lambda\):
  \[
    A_\lambda \coloneqq \{v\in A\mid \exists m\ge 0,\ g_\lambda(t)^m\cdot v = 0 \}.
  \]
The classification of modules over a PID shows that \(A\) splits as a direct sum of submodules of the form \(A_\lambda\). When this decomposition is applied to the MHS on \(\Tors_R H_i(U^f, \Q)\), one obtains a decomposition as a direct sum of MHS.
\end{theorem}

\begin{theorem}[{\cite[Corollary 4.2]{EvaMoises}}]\label{thm:tors1cohomology}
Consider the eigenspace with eigenvalue 1 of \(\Tors_R(H^i(U, \oL))\), with the MHS induced by the inclusion
\[
\Tors_R (H^i(U, \oL))_1 \hookrightarrow \Tors_R (H^i(U, \oL)),
\]
where \(\Tors_R (H^i(U, \oL))\) has the MHS from Definition \ref{def:mhsalexander} (whose construction involves a finite cover). Then, the inclusion (for \(m\gg 1\))
\[
\Tors_R(H^i(U, \oL))_1 \hookrightarrow 
H^i(U, R_{m}\otimes \oL)
\]
is a MHS morphism, where \(H^i(U, R_{m}\otimes \oL)\) has the MHS from Definition \ref{def:endowedMHS}.
\end{theorem}

The theorem above allows for the construction of the MHS on \(
\Tors_R(H^i(U, \oL))_1\) without the use of a finite cover. The theorem below allows translating this conclusion to homology as well.

\begin{theorem}[{\cite[Corollary 3.6]{EvaMoises}}]\label{thm:tors1homology}
The isomorphism \eqref{eq:torsionSide-general} in Proposition \ref{prop:ResUCT} restricts to the eigenspace with eigenvalue 1 of each component, to form the following chain of isomorphisms:
      \begin{equation}\label{eq:torsionSide}
    \Hom_\Q(\TorsH_1, \Q) \xleftarrow[\cong]{\Res}
\Ext^1_R(\TorsH_1, R)
\xrightarrow[\cong]{\mathrm{UCT}}
    \Tors_R H^{i+1}(U, \oL)_1.
  \end{equation}
  Furthermore, \(\TorsH_1\) is a sub-MHS of \(\TorsH\), and the map above is a MHS isomorphism between its dual MHS and the MHS from Theorem \ref{thm:tors1cohomology}.
\end{theorem}

\section{Auxiliary mixed Hodge complexes and structures}\label{sec:aux}

This section is devoted to endowing the cohomology of certain complexes of sheaves with canonical MHS using mixed Hodge complexes of sheaves. These new MHS will be used in Section~\ref{sec:main} to prove the main theorem (Theorem~\ref{thm:main}).

\begin{definition}\label{def:quotientsFm}
Let $m\geq 1$, and let $\cF^\bullet_m$ be the two-term complex $\left(R_\infty\otimes_R\oL\to s^{-m}R_\infty\otimes_R\oL\right)$ in degrees $-1$ and $0$, where the map between the two terms is the natural inclusion. For all $m'\geq 0$, consider its subcomplex $s^{m'}R_\infty\otimes_R\oL\xrightarrow{\Id} s^{m'}R_\infty\otimes_R\oL$. We define $\wt\cF^{\bullet}_{m,m'}$ as the quotient of $\cF^\bullet_m$ by this subcomplex, namely
$$
\wt\cF^\bullet_{m,m'}\coloneqq\left\{\begin{array}{ccr} \left(R_{m'}\otimes_R\oL\hookrightarrow \frac{s^{-m}R_\infty}{s^{m'}R_\infty}\otimes_R\oL\right)&\text{(in degrees $-1$ and $0$)} & \text{if $m'\geq 1$,}\\
\frac{s^{-m}R_\infty}{R_\infty}\otimes_R\oL &\text{(in degree $0$)}  & \text{if $m'= 0$,}
\end{array}\right.
$$
\end{definition}

\begin{remark}\label{rem:quotientsFm}
Note that the quotient maps are an infinite sequence of quasi-isomorphisms:  \[
       \cF^\bullet_m \to \cdots\to \wt\cF^\bullet_{m, m' + 1} \to \wt\cF^\bullet_{m, m'} \to \cdots \to \wt\cF^\bullet_{m, 0}.
    \] 
Indeed, the cokernel of $R_{m'}\otimes_R\oL\hookrightarrow \frac{s^{-m}R_\infty}{s^{m'}R_\infty}\otimes_R\oL$ is $\frac{s^{-m}R_\infty}{R_\infty}\otimes_R\oL$ for all $m'\geq 1$, the same as the cokernel of $R_{\infty}\otimes_R\oL\hookrightarrow s^{-m}R_\infty\otimes_R\oL$.
\end{remark}

Our next goal is to endow the hypercohomology groups of $\wt\cF^{\bullet}_{m,m'}$ with canonical MHS for all $m'\geq 0$ that are preserved by the quotient maps from Remark~\ref{rem:quotientsFm}, that is, such that those quotient maps induce MHS isomorphisms. For this, we will have to find a realization of $\wt\cF^{\bullet}_{m,m'}$ by a mixed Hodge complex of sheaves.

\begin{lemma}\label{lem:cone}
    Let $m',m\geq 1$, and let \(\sK^\bullet_{m''}\) be as in Theorem~\ref{thm:MHC} for all $m''\in\Z\setminus\{0\}$. The morphism $ \sK^\bullet_{m'}(1)
    \xrightarrow{s^m}
    \sK^\bullet_{m+m'}(1-m)$ induced by multiplication by $s^m$ from $R_{m'}\to R_{m+m'}$ is a morphism of mixed Hodge complexes of sheaves, where $(1)$ and $(1-m)$ denote Tate twists.
\end{lemma}
\begin{proof}
Let us use the isomorphisms (resp. morphism) of mixed Hodge complexes of sheaves from Theorem~\ref{thm:MHC}~\eqref{part:dual} (resp. Part~\eqref{part:inclusionProjection}, with a Tate twist by \((-m)\)). The rational part of these isomorphisms (resp. morphism) are the vertical (resp. top horizontal) arrows in Diagram~\eqref{eq:defOfsC} below. Note that the vertical arrows agree with multiplication by \(s^m\), by Remark~\ref{rem:RmDualMultiplication}. The dashed arrow below is the only morphism $\sK_m^\bullet(1)\to \sK^\bullet_{m+m'}(1-m)$ that makes the diagram commute, and, since the vertical arrows are the rational part of isomorphisms of mixed Hodge complexes of sheaves, the dashed arrow is the rational part of a morphism of mixed Hodge complexes of sheaves. We want to see that the multiplication by $s^m$ morphism in the statement of this lemma coincides with the dashed arrow below. This can be checked on the following bases: If we consider the bases \(\{1,s,s^2,\ldots \}\) of \(R_{m'}\) and \(R_{m+m'}\), and their corresponding dual bases \(\{1^\vee, s^\vee, \ldots ,\}\), the diagram above maps the elements as pictured in the right side of the diagram:
    \begin{equation}\label{eq:defOfsC}
      \begin{tikzcd}[row sep = 1.4em]
        \sK^\bullet_{-m'}(m') 
        \arrow[r, hookrightarrow]
        \arrow[d, "\cong"', "\bD_{m'}^{-1}(m')"]
        &
        \sK^\bullet_{-m-m'}(m') 
        \arrow[d, "\cong"', "\bD_{m+m'}^{-1}(m')"]
        &
        (s^i)^\vee
        \arrow[r, mapsto]
        \arrow[d, mapsto]
        &
        (s^i)^\vee
        \arrow[d, mapsto]
        \\
        \sK^\bullet_{m'}(1)
\arrow[r, dashrightarrow]
 &
    \sK^\bullet_{m+m'}(1-m)
 &
        s^{m' - 1 - i}
\arrow[r, dashrightarrow]
 &
 s^{m+m'-1-i}.
      \end{tikzcd}
    \end{equation}
\end{proof}

\begin{definition}[The mixed cone]\label{def:mixedCone}
Let $m\geq 1$ and let $m'\geq 0$. If $m'\geq 1$, let \(\sC^\bullet_{m,m'}\) be the cone of \(
    \sK^\bullet_{m'}(1)
    \xrightarrow{s^m}
    \sK^\bullet_{m+m'}(1-m)
    \), as in Definition~\ref{def:mappingCone}. In particular, as a sheaf, and not taking into account filtrations or the differential,
    $$
    \sC^p_{m,m'}=\sK^{p+1}_{m'}\oplus \sK^{p}_{m+m'} \quad\text{for all $p$.}
    $$
Since, by Lemma~\ref{lem:cone}, \(
    \sK^\bullet_{m'}(1)
    \xrightarrow{s^m}
    \sK^\bullet_{m+m'}(1-m)
    \) is the rational part of a morphism of mixed Hodge complexes of sheaves given by multiplication by $s^m$, \(\sC^\bullet_{m,m'}\) is also the rational part of a mixed Hodge complex of sheaves, namely of the mixed cone of multiplication by $s^m$, by Theorem \ref{thm:mixedCone}.

    If $m'=0$, we let $\sC^\bullet_{m,0}\coloneqq \sK_m^\bullet(1-m)$, which is the rational part of a mixed Hodge complex of sheaves (and can also be interpreted as the cone of $0\to \sK_m^\bullet(1-m)$ following the definition of $\sC^\bullet_{m,m'}$ for $m'\geq 1$).
\end{definition}

\begin{remark}\label{rem:lesMixedCone}
    By Theorem~\ref{thm:mixedCone}, the short exact sequence of complexes of sheaves
    $$
    0\to \sK^\bullet_{m+m'}(1-m)\to \sC^\bullet_{m,m'}\xrightarrow{c} \sK^\bullet_{m'}[1]\to 0
    $$
    (where the connecting morphism $c$ is given by the projection onto the first summand, as in Definition \ref{def:mappingCone}) induces a long exact sequence of morphisms of MHS
    $$
    \cdots\to \mathbb{H}^i(U,\sK^\bullet_{m+m'})(1-m)\to \mathbb{H}^i(U,\sC^\bullet_{m,m'})\to \mathbb{H}^{i+1}(U,\sK^\bullet_{m'})(1)\to\cdots
    $$
\end{remark}

\begin{defprop}[The MHS on $\mathbb{H}^i(U,\wt\cF^{\bullet}_{m,m'})$ and $\mathbb{H}^i(U,\cF^{\bullet}_{m})$]\label{defprop:aux}
Let $m\geq 1$ and $m'\geq 0$. Let $\wt\cF^{\bullet}_{m,m'}$ be the complex of sheaves from Definition~\ref{def:quotientsFm}. Let \(\sC^\bullet_{m,m'}\) be the complex of sheaves from Definition~\ref{def:mixedCone}, which is the rational part of a mixed Hodge complex of sheaves. Let $\nu_{m''}$ be the morphism from Theorem~\ref{thm:MHC}~\eqref{part:nu} for all $m''\in\Z\setminus\{0\}$, and let $\pairing_m$ be the morphism from Remark~\ref{rem:RmDual}. Then, the following hold:
\begin{enumerate}
    \item\label{part:diagramTwoRows} Suppose that $m'\geq 1$. Then, the following diagram commutes.
    \begin{equation}\label{eq:Nu}
      \begin{tikzcd}[row sep =1.4em, column sep = 5em]
        R_{m'} \otimes_R \oL 
      \arrow[d, hookrightarrow]
      \arrow[r, equals]
      &
      R_{m'} \otimes_R \oL 
      \arrow[d, "s^m"]
      \arrow[r, "\nu_{m'}"]
      &
      j^{-1} \sK^\bullet_{m'}(1)
      \arrow[d, "s^m"]
      \\
      \displaystyle\frac{s^{-m}R_\infty}{s^{m'} R_\infty} \otimes_R \oL
      \arrow[r, "\cong"', "s^{m}"]
      & 
      R_{m+m'} \otimes_R \oL 
      \arrow[r, "\nu_{m+m'}"']
      &
      j^{-1} \sK^\bullet_{m+m'}(1-m).
    \end{tikzcd}
    \end{equation}
    \item\label{part:quisos} The horizontal arrows in diagram~\eqref{eq:Nu} are all quasi-isomorphisms. The same holds for the second row of~\eqref{eq:Nu} if $m'=0$. 
    \item\label{part:coneColumn} Suppose that $m'\geq 1$. If we interpret the first column of diagram~\eqref{eq:Nu} as a morphism of complexes of sheaves concentrated at degree $0$, the cone of this morphism coincides with $\wt\cF_{m,m'}^\bullet$.  The cone of the last column of diagram~\eqref{eq:Nu} is $j^{-1}\sC_{m,m'}^\bullet$.
    \item\label{part:defNu}

    Suppose that $m'\geq 1$. Diagram \eqref{eq:Nu} induces a quasi-isomorphism between the cones of the first and last columns, which we denote by
    $$
    \mathcal{V}_{m,m'}\colon \wt\cF_{m,m'}^\bullet\to j^{-1}\sC_{m,m'}^\bullet.
    $$
    In degree $-1$ it coincides with the map  $\nu_{m'}$, and in degree $0$ it corresponds to $0$ in the first summand of $j^{-1}\sC_{m,m'}^0=j^{-1}\sK^1_m\oplus j^{-1}\sK^0_{m+m'}$ and \(\nu_{m+m'}\circ s^m\) in the second summand.
    \item\label{part:defNu0} If $m'=0$, the quasi-isomorphism $\mathcal{V}_{m,0}\colon \wt\cF_{m,0}^\bullet\to j^{-1}\sC_{m,0}^\bullet$, is defined as \(\nu_{m}\circ s^m\). This agrees with the interpretation for $m'\geq 1$ if we switch diagram~\eqref{eq:Nu} for a diagram in which the first row is equal to $0$ and the second row remains the same (making $m'=0$).
    \item\label{part:defMHSquotient} The sequence of quasi-isomorphisms
    $$Rj_*\wt\cF^\bullet_{m,m'}\xrightarrow[\cong]{Rj_*\mathcal{V}_{m,0}}j^{-1}Rj_*j^{-1}\sC_{m,m'}^\bullet\xleftarrow[\cong]{\text{adjunction }\Id\to Rj_*j^{-1}}\sC_{m,m'}^\bullet$$ endows $\mathbb{H}^i(U,\wt\cF^\bullet_{m,m'})$ with canonical MHS for all $i$.
    \item\label{part:defMHSF} The quotient maps from~\eqref{rem:quotientsFm} induce isomorphisms of MHS in cohomology
    $$
    \cdots \xrightarrow{\cong}\mathbb{H}^i(U,\wt\cF^\bullet_{m,3})\xrightarrow{\cong}\mathbb{H}^i(U,\wt\cF^\bullet_{m,2})\xrightarrow{\cong}\mathbb{H}^i(U,\wt\cF^\bullet_{m,1})\xrightarrow{\cong}\mathbb{H}^i(U,\wt\cF^\bullet_{m,0}).
    $$
    This sequence endows $\mathbb{H}^i(U,\cF^\bullet_{m})$ with a canonical MHS.
\end{enumerate}
\end{defprop}

\begin{proof}
\begin{enumerate}
    \item The left hand square clearly commutes. The second square commutes because it fits into the following diagram, by Theorem \ref{thm:MHC}\eqref{part:KmfromK1}:
    \[
    \begin{tikzcd}[row sep = 1.5em]
        R_{m+m'}\otimes_R \oL
        \arrow[r,twoheadrightarrow]
        \arrow[d, "\Id\otimes \nu_{m+m'}"]
        &
        R_{m'}\otimes_R \oL
        \arrow[r, hookrightarrow, "s^m"]
        \arrow[d, "\Id\otimes \nu_{m+m'}"]
        &
        R_{m+m'}\otimes_R \oL
        \arrow[d, "\Id\otimes \nu_{m+m'}"]
        \\
        R_{m+m'}\otimes_R \sK^\bullet_{1}
        \arrow[r,twoheadrightarrow]
        &
        R_{m'}\otimes_R \sK^\bullet_{1}
        \arrow[r, hookrightarrow, "s^m"]
        &
        R_{m+m'}\otimes_R \sK^\bullet_{1}.
        \end{tikzcd}
    \]
    Note that both horizontal compositions are multiplication by \(s^m\), so the outer rectangle commutes by Theorem \ref{thm:MHC}  \eqref{part:multiplication}. The left hand square commutes by Theorem \ref{thm:MHC} \eqref{part:inclusionProjection}, so the right hand square must commute, taking into account that the left horizontal arrows are surjective and the right horizontal arrows are injective.
    \item The arrow \(s^m\) is induced by an $R_\infty$-module isomorphism on the first factors of the tensor products, so it is an isomorphism of sheaves. For the remaining arrows, this is the content of Theorem~\ref{thm:MHC}~\eqref{part:nu}.
    \item This is immediate from the definitions.
    \item The precise description of $\mathcal{V}_{m,m'}$ follows from the definition. Since both rows are quasi-isomorphisms, the five lemma implies that the map induced between the cones of both rows is also a quasi-isomorphism.
    \item $\mathcal{V}_{m,0}$ is a quasi-isomorphism by part~\eqref{part:quisos}.
    \item We need to show that the adjunction $\Id\to Rj_*j^{-1}$ applied to the complex of sheaves $\sC^\bullet_{m,m'}$ yields a quasi-isomorphism. By Theorem~\ref{thm:MHC}~\eqref{part:toX}, the adjunction induces quasi-isomorphisms \(\sK_{m'}^\bullet\cong R j_*j^{-1}\sK_{m'}^\bullet \) and \(\sK_{m+m'}^\bullet\cong R j_*j^{-1}\sK_{m+m'}^\bullet\), so the five lemma implies that \(\sC^\bullet_{m,m'}\cong Rj_*j^{-1} \sC^\bullet_{m,m'}\). Therefore, \(\sC^\bullet_{m,m'}\) endows the hypercohomology of \(\wt \cF^\bullet_{m,m'}\) with a MHS, via the isomorphisms:
    \[
    \bH^i(U, \wt \cF^\bullet_{m,m'})
    \overset{\cV_{m,m'}}\cong 
        \bH^i(U, j^{-1}\sC^\bullet_{m,m'})
    \cong 
        \bH^i(X, Rj_*j^{-1}\sC^\bullet_{m,m'})
    \cong 
        \bH^i(X,\sC^\bullet_{m,m'}).
    \]
    \item Let us show that for different \(m'\) the quasi-isomorphisms \(\cV_{m,m'}\) are compatible, in the sense that if \(m''\ge m'\), the quotient map \(\wt\cF^\bullet_{m, m''}\to \wt \cF^\bullet_{m, m'}\) induces the obvious quotient map \(j^{-1}\sC^\bullet_{m,m''}\to j^{-1}\sC^\bullet_{m,m'}\). We first do this for the rows of~\eqref{eq:Nu} separately. For the second row of~\eqref{eq:Nu}, we must check that the following commutes:
    \[
    \begin{tikzcd}[row sep =1.4em, column sep = 4.5em]
      \frac{s^{-m}R_\infty}{s^{m''} R_\infty} \otimes_R \oL
      \arrow[d, twoheadrightarrow, "\text{projection}"]
      \arrow[r, "\cong"', "s^{m}"]
      &
      R_{m+m''} \otimes_R \oL
      \arrow[r, "\nu_{m+m''}"]
      \arrow[d, twoheadrightarrow, "\text{projection}"]
      &
      j^{-1} \sK^\bullet_{m+m''}(1-m)
      \arrow[d, twoheadrightarrow, "\text{projection}"]
      \\
      \frac{s^{-m}R_\infty}{s^{m'} R_\infty} \otimes_R \oL
      \arrow[r, "\cong"', "s^{m}"]
      &
      R_{m+m'} \otimes_R \oL
      \arrow[r, "\nu_{m+m'}"]
      &
      j^{-1} \sK^\bullet_{m+m'}(1-m).
    \end{tikzcd}
    \]
    It is clear for the left hand square, and the right hand square is again Theorem \ref{thm:MHC}\eqref{part:inclusionProjection}. Moreover, note that the projection $\sK^\bullet_{m+m''}(1-m)\to \sK^\bullet_{m+m'}(1-m)$ is the rational part of a morphism of mixed Hodge complexes of sheaves by Theorem~\ref{thm:MHC}~\eqref{part:inclusionProjection}. The proof for the first row of~\eqref{eq:Nu} is the same, as it is the particular case of \(m=0\), and also note that the projection $\sK^\bullet_{m''}(1)\to \sK^\bullet_{m'}(1)$ is the rational part of a morphism of mixed Hodge complexes of sheaves by Theorem~\ref{thm:MHC}~\eqref{part:inclusionProjection}. Since both rows commute with the quotient maps, the corresponding map of cones commutes as well, and the corresponding map between $j^{-1}\sC^\bullet_{m,m''}$ and $j^{-1}\sC^\bullet_{m,m'}$ is $j^{-1}$ applied to the natural projection between $\sC^\bullet_{m,m''}$ and $\sC^\bullet_{m,m'}$, which is the rational part of a morphism of mixed Hodge complexes of sheaves. We have shown that the sequence \(\wt\cF^\bullet_{m,m'}\) corresponds through the quasi-isomorphisms $\mathcal{V}_{m,m'}$ to $j^{-1}$ applied to the following chain of natural projections,
    \[
      \cdots \twoheadrightarrow  \sC^{\bullet}_{m,2} \twoheadrightarrow 
      \sC^{\bullet}_{m,1}  \twoheadrightarrow  \sC^\bullet_{m,0} 
    \]
    and all of these projections are the rational part of a morphism of mixed Hodge complexes of sheaves, so they induce MHS morphisms in cohomology. In fact, these MHS morphisms will be isomorphisms by Remark~\ref{rem:quotientsFm}.
\end{enumerate}
\end{proof}

  \section{Proof of the main theorem}\label{sec:main}

\begin{theorem}\label{thm:main}
  Let \(U\), \(f\) and \(U^f\) be as in Definition \ref{def:Uf}. Let \(m\ge 1\). Let \(R = \Q[\pi_1(\C^*)]\cong \Q[t^{\pm 1}]\). Consider the MHS on \(\Tors_R H_i(U^f,\Q)\) from Definition~\ref{def:mhsalexander}, and the MHS on \(\frac{H_i(U^f, \Q)}{(t-1)^m H_i(U^f, \Q)}\) from Definition~\ref{def:proMHS}. Since they are a submodule and a quotient of the Alexander module, respectively, there is a natural composition map:
  \begin{equation}\label{eq:mainThm}
    \Tors_R H_i(U^f, \Q) \hookrightarrow H_i(U^f, \Q) \twoheadrightarrow \frac{H_i(U^f, \Q)}{(t-1)^mH_i(U^f, \Q)}.
  \end{equation}
  This composition map is a MHS morphism for all $m\geq 1$.
\end{theorem}

In order to prove Theorem~\ref{thm:main}, it will be necessary to understand the following morphism.
\begin{definition}[Definition of $\Xi$]\label{def:Xi}
The morphism
  \[
   \Xi\colon\Hom_\Q\left(\frac{\Alex}{(t-1)^m\Alex},\Q\right)
   \to
 \Hom_\Q\left(\TorsH_1, \Q\right).
  \]
  is defined as the $\Q$-dual of the restriction of the map in~\eqref{eq:mainThm} to \(\TorsH_1\) for all $i$.
\end{definition}

Specifically, we will need to show that $\Xi$ is a MHS morphism. We will not do this directly, instead we will relate $\Xi$ to other morphisms that will be easier to understand. We start by defining these auxiliary morphisms.

\begin{definition}[Definition of $\wt\Xi$]\label{def:wtXi}
Let $m\geq 1$ and let $\cF_m^\bullet\coloneqq\left(R_{\infty}\otimes_R\oL\hookrightarrow s^{-m}R_{\infty}\otimes_R\oL \right)$ be a two-term complex in degrees $-1$ and $0$ (as in Definition~\ref{def:quotientsFm}). We define the morphism
$$\wt\Xi\colon \cF_m^\bullet \to R_\infty \otimes_R \oL[1]$$
as \(-\Id_{R_\infty \otimes_R \oL}\) in degree \(-1\) and vanishing in degree \(0\).
\end{definition}

\begin{definition}[Definition of $p$]\label{def:p}
Let $m\geq 0$. We define the epimorphism
$$
p:\Hom_\Q(H_i(U,R_m\otimes_R\cL),\Q)\twoheadrightarrow\Hom_\Q\left(\frac{H_i(U^f,\Q)}{(t-1)^mH_i(U^f,\Q)},\Q\right)
$$
as the $\Q$-dual of the map
\begin{equation}\label{eq:pdual}
\frac{H_i(U,\cL)}{(t-1)^m H_i(U,\cL)} \to H_i(U, R_m\otimes_R \cL)\end{equation}
induced in homology by the projection $\cL\twoheadrightarrow R_m\otimes_R\cL$, under the identification 
 \(H_i(U^f, \Q)\cong H_i(U, \cL)\) from Proposition \ref{prop:UfvscL}. Next, we consider the map induced by the projection \(\cL\twoheadrightarrow R_m\otimes_R \cL\), which, in homology, factors through
$$
H_i(U, \cL)\to
\frac{H_i(U,\cL)}{(t-1)^m H_i(U,\cL)} \to H_i(U, R_m\otimes_R \cL).$$
\end{definition}

Indeed, $p$ is surjective because the morphism in \eqref{eq:pdual} is injective by Proposition~\ref{prop:TorsIntoHomology}.

\begin{definition}[Definition of $\pairing_m$]\label{def:pairingSheaves}
 Let $m\geq 1$ and let $\cF_m^\bullet$ denote the same complex as in Definition~\ref{def:wtXi}.  Let $\pairing_m$ be as in Remark~\ref{rem:RmDual}. With a slight abuse of notation, we define the quasi-isomorphism
 $$
\pairing_m:\cF^\bullet_m\to R_{-m}\otimes \oL
 $$
 as vanishing in degree $-1$ and induced by the map
 $$
    \begin{array}{ccc}
     s^{-m}R_\infty    & \longrightarrow &  R_{-m}\\
       g(s)  &\longmapsto &\langle \ov{g(s)},\cdot\rangle_m
    \end{array}$$
    in degree $0$, where $\ov{g(s)}$ denotes the image of $g(s)$ through the projection $s^{-m}R_\infty\to \frac{s^{-m}R_\infty}{R_\infty}$.
\end{definition}

In this section we will prove Theorem~\ref{thm:main} in a series of lemmas, the first of which specifies the precise relation between $\Xi$ and the other morphisms that we have just defined. Note that the statement of the following result is not related to Hodge theory.

  \begin{lemma}\label{lem:commutes}
    Let $\cF_m^\bullet$ be as in Definitions~\ref{def:quotientsFm} and~\ref{def:wtXi}. Consider the maps $\Theta^\vee$, $\Xi$, $\wt\Xi$, $p$, $\pairing_m$ and \(\mathrm{UCT}\circ \Res^{-1} \) from Definitions~\ref{def:endowedMHSHomology}, \ref{def:Xi}, \ref{def:wtXi}, \ref{def:p} and \ref{def:pairingSheaves} and Theorem \ref{thm:tors1homology}, respectively. Then, the following diagram commutes:
    \begin{equation}\label{eq:bigComposition}
      \begin{tikzcd}[column sep = 1.2em, row sep = 1em]
H^i(U, R_{-m}\otimes_R \oL)
\arrow[r, twoheadrightarrow, "p\circ \Theta^\vee"]
&
\displaystyle 
\Hom_\Q\left(\frac{\Alex}{(t-1)^m\Alex},\Q\right)
\arrow[r, "\Xi"]
&
\Hom_\Q(\TorsH_1, \Q)
\arrow[d, "\mathrm{UCT}\circ \Res^{-1}", "\cong"']
\\
    \bH^i(U, \cF_m^\bullet)
    \arrow[u, "\langle \cdot{,} \cdot \rangle_m", "\cong"']
\arrow[r, "\wt\Xi"]
&
R_\infty \otimes_R H^{i+1}(U, \oL)
&
\Tors_R H^{i+1}(U, \oL)_1.
\arrow[l, hookrightarrow]
  \end{tikzcd}
  \end{equation}
  \end{lemma}
  \begin{proof}
    Let us compute the composition of the top row of \eqref{eq:bigComposition}: it is given by composing \(\Theta^\vee\) with \(\Xi\circ p\). The composition \(\Xi\circ p\) is the dual of the map induced in homology by the projection \(\cL\to R_m\otimes_R \cL\), restricted to $\Tors_R H_i(U^f,\Q)_1$. Applying Theorem \ref{thm:dualHomology}~\eqref{part:functorialPairing}, the top row of \eqref{eq:bigComposition} is given 
by applying Theorem \ref{thm:dualHomology}, Parts~\eqref{part:functorialPairing} and~\eqref{part:naivePairingInHomology} to the pairing
     \[
      \begin{tikzcd}[row sep = 0em, column sep =0]
        \left(R_{-m}\otimes_R \oL\right)
        &
        \times
      &
  \cL
      \arrow[rr]
      &&
      \ul\Q
      \\
      \quad\quad\quad(\phi \otimes a
      &
      {,}
      &
      b)
      \arrow[rr, mapsto]
      &\phantom{a}&
      \phi\left(\ov{a(b)}\right).
    \end{tikzcd} 
    \]
    Here, \(\ov {a(b)}\) denotes the image of \(a(b)\in R\) inside \(R_m\) (see Definition \ref{def:exp}).
    
Next, let us consider the map \(\langle \cdot, \cdot\rangle_m\). If \(g\in s^{-m}R_{\infty}\), its image in \(R_{-m}\) is the element \(\langle g, \cdot \rangle_m\in R_{-m}\), by Remark~\ref{rem:RmDualMultiplication}. 
    Applying Theorem \ref{thm:dualHomology}~\eqref{part:functorialPairing} again, the composition
\(\Xi\circ p\circ \Theta^\vee \circ\langle\cdot,\cdot\rangle_m\) is given by the pairing that vanishes on \(\cF^{-1}\), and equals the following on \(\cF^0\):
\begin{equation}\label{eq:pairingFirstHalf}
      \begin{tikzcd}[row sep = 0.2em, column sep = 0.0em]
        \left(s^{-m}R_{\infty}\otimes_R \oL\right)
        &
        \times
      &
       \cL
      \arrow[rr]
      &&
      \ul\Q
      \\
      \quad\quad\quad(g \otimes a
      &
      {,}
      &
       b)
      \arrow[rr, mapsto]
      &\phantom{a}&
      \langle g, \ov{a(b)} \rangle_m = \res_0 (a(b)\cdot g).
    \end{tikzcd} 
\end{equation}
Now that we have understood the composition \(\Xi\circ p\circ \Theta ^\vee\circ\langle\cdot,\cdot\rangle_m\) as a map induced by Theorem~\ref{thm:dualHomology}~\eqref{part:functorialPairing} from a pairing, let us focus on the rest of the maps in the diagram \eqref{eq:bigComposition}. We will start by trying to understand the map \(\Res\circ\mathrm{UCT}^{-1}\), more specifically, the map $\Res$, as defined in Theorem~\ref{thm:tors1homology}. Recall that $\Res$ is defined as the composition (from left to right)
    \begin{equation}\label{eq:Reslemma}
      \Hom_\Q(A, \Q)
      \xleftarrow[\cong]{\res_*}
       \Hom_R(A, K/R) 
      \xrightarrow[\cong]{K/R \to R[1]}
      \Ext^1_R(A, R)
    \end{equation}
    applied to $A=\Tors_R H_i(U^f,\Q)_1$.

    With this description, diagram~\eqref{eq:bigComposition} becomes:
        \begin{equation}\label{eq:bigCompositionres}
      \begin{tikzcd}[column sep = 1.3em, row sep = 1em]
\Hom_\Q(\Tors_R H_i(U^f,\Q)_1,\Q)
&\Hom_R(\Tors_R H_i(U^f,\Q)_1,K/R)\arrow[l,"\res_*"', "\cong"]\arrow[r,"K/R\to R{[1]}"',outer sep=5pt]
\arrow[r, phantom, "\scriptstyle \cong", shift left = 0.5em]
&\Ext_R^1(\Tors_R H_i(U^f,\Q),R)_1\arrow[d,"\mathrm{UCT}"]\\
\mathbb{H}^i(U,\cF^\bullet_m)\arrow[u, "\Xi\circ p\circ\Theta^\vee\circ\pairing_m"]\arrow[r,"\wt\Xi"]
& R_\infty\otimes_R H^{i+1}(U,\oL)
&\Tors_R H^{i+1}(U,\oL)_1\arrow[l,hook]
  \end{tikzcd}
  \end{equation}

    By Proposition~\ref{prop:TorsIntoHomology}, the natural projection $\cL\twoheadrightarrow R_m\otimes_R\cL$ induces monomorphisms $$\Tors_R H_i(U, \cL)_1 \hookrightarrow H_i(U, R_{m}\otimes_R \cL)$$ for all \(m \gg 1\).  Furthermore, note that all three functors \(\Ext^1_R(\cdot, R)\), \(\Hom_R(\cdot, K/R)\) and \(\Hom_\Q(\cdot, \Q)\) are right-exact, so the injection will turn into a dual surjection. The map~\eqref{eq:Reslemma} above is induced, via this surjection for \(m\gg 1\), by the analogous maps applied to \(A=H_i(U, R_m\otimes \cL)\):
\begin{equation*}
      \Hom_\Q(H_i(U, R_{m}\otimes_R \cL), \Q)
\xleftarrow{\cong}
      \Hom_R(H_i(U, R_{m}\otimes_R \cL), K/R)
      \xrightarrow{\cong}
      \Ext^1_R(H_i(U, R_{m}\otimes_R \cL), R).
    \end{equation*}
    Let us start by focusing on the left hand map. Applying Theorem~\ref{thm:dualHomology}~\eqref{part:perfectPairing}, we obtain the following commutative diagram (note that \(K/R\) is injective):
    \[
      \begin{tikzcd}[row sep = 1.2em]
        \Hom_\Q(H_i(U, R_{m}\otimes_R \cL), \Q)
&
\arrow[l, "\cong", "\res_*"']
\Hom_R(H_i(U, R_{m}\otimes_R \cL), K/R) 
\\
H^{i}(U, \Homm_\Q (R_{m}\otimes_R \cL, \Q))
\arrow[u, "\text{Theorem \ref{thm:dualHomology}}\eqref{part:perfectPairing}", "\cong"', outer sep = 0.3em]
&
H^{i}(U, \Homm_R (R_{m}\otimes_R \cL, K/R)).
\arrow[l, "\cong", "\res_*"']
\arrow[u, "\text{Theorem \ref{thm:dualHomology}}\eqref{part:perfectPairing}", "\cong"', outer sep = 0.3em]
\end{tikzcd}
    \]
    The bottom row map is induced in cohomology by the sheaf morphism \(\res_*\) (as a map of constant sheaves). 
 We include this map into the following diagram, and claim that it commutes,
    \begin{equation}\label{eq:resExpanded}
      \begin{tikzcd}[column sep = 5em, row sep = 1.2em]
      \cF^\bullet_m
      \arrow[r]
      \arrow[d, equals]
      &
      \frac{s^{-m}R_\infty}{R_\infty}\otimes_R \oL
      \arrow[d, "\pairing_m\otimes \Id_{\oL}"]
      \arrow[r, "\mathrm{mult}"]
      &
      \Homm_R(R_{m}\otimes_R \cL, K/R)
      \arrow[r]
      \arrow[d, "\res_*", "\cong"']
      &
      \Homm_R(\cL, K/R)
      \arrow[d, "\res_*", "\cong"']
      \\
      \cF^\bullet_m
      \arrow[r, "\pairing_m"]
      &
      R_{-m} \otimes_R \oL
      \arrow[r]
      &
      \Homm_\Q(R_{m}\otimes_R\cL, \Q)
      \arrow[r]
      &
      \Homm_\Q(\cL, \Q).
    \end{tikzcd}
    \end{equation}
    In this diagram:
    \begin{itemize}
        \item The top left horizontal arrow is given by the projection $\cF^0_m=s^{-m}R_\infty\otimes_R\oL\twoheadrightarrow \frac{s^{-m}R_\infty}{R_\infty}\otimes_R\oL$.
        \item \(\pairing_m\) in the arrow labeled by \(\pairing_m\otimes\Id_{\cL}\) denotes the isomorphism $\frac{s^{-m}R_\infty}{R_\infty}\to R_{-m}$ from Remark~\ref{rem:RmDual}.
        \item The second horizontal arrow from the left in the bottom row is the one induced by the pairing in \eqref{eq:pairingTruncated}.
        \item The arrow labeled as $\mathrm{mult}$ is induced from the pairing $$
        \begin{array}{cccc}
\mathrm{mult}\colon&\left(\frac{s^{-m}R_\infty}{R_\infty} \otimes_R \oL \right)\times \left(R_m\otimes_R \cL \right)&\longrightarrow &K/R\\
&(g(s)\otimes a, \alpha(s)\otimes b) &\longmapsto& g(s)\cdot \left( \alpha(s)\cdot a(b)\right).
\end{array}$$
Let us break down this definition: $\alpha(s)\cdot a(b)$ is an element in $R_m$ (recall that $R_m$ has an $R$-module structure by Definition~\ref{def:exp}). Hence, $g(s)\cdot \left( \alpha(s)\cdot a(b)\right)\in \frac{s^{-m}R_\infty}{R_\infty}$. Identifying $\frac{s^{-m}R_\infty}{R_\infty}$ with $\frac{(t-1)^{-m}R}{R}$ through the isomorphism induced by the monomorphism $R\hookrightarrow R_\infty$ in Definition~\ref{def:exp}, $g(s)\cdot \left(\alpha(s)\cdot a(b)\right)\in \frac{(t-1)^{-m}R}{R}\subset K/R$.
    \end{itemize}
    The first square commutes by definition. It is straightforward to check that the second square commutes using Remark~\ref{rem:comparisonResidues}. The commutativity of the third square is clear, since both vertical arrows are induced by the same map \(\res\colon K/R\to \Q\). Note that the composition of the bottom row of~\eqref{eq:resExpanded} is the pairing~\eqref{eq:pairingFirstHalf}. The composition of the top row of~\eqref{eq:resExpanded} is induced by the $R$-bilinear pairing
    $$
    \begin{array}{ccc}
\left(s^{-m}R_{\infty}\otimes_R \oL\right)\times \cL&\longrightarrow & K/R\\
 \quad\quad\quad (g(s)\otimes a\,,\, b) &\longmapsto &a(b)\cdot g(s),
    \end{array}
    $$
where $a(b)\cdot g\in s^{-m}R_\infty$ is seen in the quotient $\frac{s^{-m}R_\infty}{R_\infty}$, which, as in the definition of $\mathrm{mult}$, is identified with $\frac{(t-1)^{-m}R}{R}\subset K/R$.

Note that the bottom row of \eqref{eq:resExpanded} (which is induced by the same pairing as $\Xi\circ p\circ\Theta^\vee\circ\pairing_m$) composed with $\res_*^{-1}$ equals the top row of \eqref{eq:resExpanded}, and compare that to \eqref{eq:bigCompositionres}. To prove the lemma, it remains to show that the following diagram commutes, where the horizontal arrow at the top is given by the pairing corresponding to the top row of~\eqref{eq:resExpanded}:
\[
\begin{tikzcd}[column sep = 4em]    &
    \bH^i(U,\cF^\bullet_m)
    \arrow[dl, "\wt\Xi"']
    \arrow[r, "(g\otimes a{,}b)\mapsto a(b) \cdot g"]
    &
    \Hom_R(\Tors_R H_i(U,\cL)_1, K/R)
    \arrow[d, "K/R \to R{[1]}"]
\\
    R_\infty \otimes_R H^{i+1}(U, \oL)
    & 
    \Tors_R
    H^{i+1}(U, \oL)_1
    \arrow[l, hookrightarrow]
    \arrow[from=r,"\mathrm{UCT}"', "\cong"]
    &
    \Ext^1_R(\Tors_R H_i(U, \cL)_1, R)
    .
\end{tikzcd}
\]
Note that all objects here are \(R_\infty\) modules (since the ones on the right side of the diagram are torsion and supported at \(t=1\)). Therefore, the diagram does not change it we tensor by \(R_\infty\). Let us denote by \(K_\infty\) the field of fractions of \(R_\infty\). We obtain the following diagram:\[
\begin{tikzcd}[column sep = 4em]    \bH^i(U,\cF^\bullet_m)
    \arrow[d, "\wt\Xi"']
    \arrow[r, "(g\otimes a{,}b)\mapsto a(b) \cdot g"]
    &
    \Hom_{R_\infty}(\Tors_{R_\infty} H_i(U,R_\infty \otimes_R \cL), {K_\infty}/{R_\infty})
    \arrow[d, "{K_\infty}/{R_\infty} \to {R_\infty}{[1]}"]
    \arrow[dl, "*"']
\\
 H^{i+1}(U, R_\infty \otimes_R \oL)
    \arrow[from=r,"\mathrm{UCT}"', "\cong"]
    &
    \Ext^1_{R_\infty}(\Tors_{R_\infty} H_i(U, R_\infty \otimes_R \cL), {R_\infty})
    .
\end{tikzcd}
\]
Here, the arrow ``$*$'' is defined as the composition of $K_\infty/R_\infty\to R_\infty[1]$ and $\mathrm{UCT}$.
Now, if we apply Corollary~\ref{cor:UCTinfty}, we have a description of the arrow \(*\), namely, it is the map \(\Hom_{R_\infty}(\Tors_{R_\infty}H_i(U,R_\infty \otimes_R L), K_\infty /{R_\infty})\to H^{i+1}(U, \Homm_{R_\infty}(R_\infty \otimes_R L, {R_\infty}))\) in the statement of said theorem applied to $L=\cL$. Using the maps defined in the statement of Theorem \ref{thm:UCT}, we see that the previous diagram commutes if and only if the following diagram commutes:
\[
\begin{tikzcd}[column sep = 6em, row sep = 1.5em]
    \bH^i(U,\cF^\bullet_m)
    \arrow[d, "\wt\Xi"']
    \arrow[r, "(g\otimes a{,}b)\mapsto a(b) \cdot g"]
    &
    \Hom_{R_\infty}(\Tors_{R_\infty} H_i(U, R_\infty \otimes_R \cL),K_\infty/R_\infty)
    \arrow[d, "K_\infty/R_\infty \to R{[1]}"]
    \\
    H^{i+1}(U, R_\infty \otimes_R \oL)
    &
    H^i(\Hom_{R_\infty}^\bullet(C_\bullet(U,  R_\infty \otimes_R \cL), R_\infty[1])) .
    \arrow[from=l, "\text{Theorem \ref{thm:dualHomology}\eqref{part:perfectPairing}}", "\cong"']
\end{tikzcd}
\]
Note that the top horizontal arrow factors through the whole module, not just the torsion quotient. Let us denote \(I_\infty^\bullet = R_\infty \otimes_R I^\bullet\), and let us also compose with the isomorphism induced by \(R_\infty \cong I_\infty^\bullet\). We obtain the following diagram (now transposed).
\begin{equation}\label{eq:quisoBroken}
\begin{tikzcd}[column sep = 4em, row sep = 1.5em]
    \bH^i(U,\cF^\bullet_m)
    \arrow[r, "\wt\Xi"']
    \arrow[d, "(g\otimes a{,}b)\mapsto a(b) \cdot g"]
    &
        H^{i+1}(U, R_\infty \otimes_R \oL)
         \arrow[d, "\text{Theorem \ref{thm:dualHomology}\eqref{part:perfectPairing}}"]
\\
    \Hom_{R_\infty}(H_i(U,R_\infty \otimes_R \cL), K_\infty/{R_\infty})
    \arrow[d, "K_\infty/{R_\infty} = I_\infty^1"]
    &
    H^i(\Hom_{R_\infty}^\bullet(C_\bullet(U, R_\infty \otimes_R \cL), {R_\infty}[1])) 
    \arrow[d, "\cong"']
    \\
H^i(\Hom_{R_\infty}^\bullet(C_\bullet(U, R_\infty \otimes_R \cL), I_\infty^\bullet[1])) 
\arrow[r, equals]
&
    H^i(\Hom^\bullet_{R_\infty}(C^\bullet(U, R_\infty \otimes_R \cL), I_\infty^\bullet[1]))
    .
\end{tikzcd}
\end{equation}
Applying Theorem \ref{thm:dualHomology}~\eqref{part:perfectPairing} (i.e. tensor-hom adjunction) we can replace the complexes in the right hand column that have the form \(\Hom_{R_\infty}^\bullet(C_\bullet(U, R_\infty \otimes_R \cL), M)\), by
\begin{align*}
 \MoveEqLeft[3] \Hom^\bullet_{{R_\infty}[\pi_1]}(C_\bullet(\wt U), \Homm_{R_\infty}(R_\infty \otimes_R \cL_x, M)) 
\\
&\cong
\Hom^\bullet_{{R_\infty}[\pi_1]}(C_\bullet(\wt U),  M \otimes_{R_\infty } (R_\infty \otimes_R \oL))  
&\text{(\(R_\infty \otimes_R \cL_x\) is free)}
\\
&= C^\bullet( U,M \otimes_{R_\infty } (R_\infty \otimes_R \oL)  )
&\text{(Proposition \ref{prop:chainHomology})}
\\
&\cong  C^\bullet( U,M \otimes_R \oL)  .
\end{align*}
We apply this to the right hand column of \eqref{eq:quisoBroken}, with \(M = R_\infty[1]\) and \(M = I_\infty^\bullet[1]\). Note that we can apply the isomorphism to \(I_\infty^\bullet\) even though it is a complex, since we just apply Theorem~\ref{thm:dualHomology}~\eqref{part:perfectPairing} for \(I^0_\infty\) and \(I^1_\infty\). We obtain the following diagram.

\[
\begin{tikzcd}[column sep = 4em, row sep = 1.5em]    \bH^i(U,\cF^\bullet_m)
    \arrow[r, "\wt\Xi"']
    \arrow[d, "(g\otimes a{,}b)\mapsto a(b) \cdot g"]
    &
        H^{i+1}(U, R_\infty \otimes_R \oL)
         \arrow[d, "\text{Theorem \ref{thm:dualHomology}\eqref{part:perfectPairing}}"]
\\
    \Hom_{R_\infty}(H_i(U,R_\infty \otimes_R \cL), K_\infty/{R_\infty})
    \arrow[d, "K_\infty/{R_\infty} = I_\infty^1"]
    &
    H^i(C^\bullet(U, R_\infty[1] \otimes_R \oL) )
    \arrow[d, "\cong"']
    \\
H^i(\Hom_{R_\infty}^\bullet(C_\bullet(U, R_\infty \otimes_R \cL), I_\infty^\bullet[1]))
\arrow[r, "\cong"]
&
    H^i(C^\bullet(U, I_\infty^\bullet[1] \otimes_R \oL) )
    .
\end{tikzcd}
\]

Next, we apply Proposition~\ref{prop:chainHomology} to conclude that the cohomology of the complexes on the bottom right is the sheaf cohomology. Note that we can apply it to \(I^\bullet_\infty\otimes_R \oL\), even though it is a complex. Indeed, the quasi-isomorphism \(I^\bullet_\infty \cong R_\infty\) induces a quasi-isomorphism \(C^\bullet(U, I^\bullet_\infty\otimes_R\oL) \cong  
C^\bullet(U, R_\infty\otimes_R\oL)\), so we have:
\[
H^i(C^\bullet(U, I^\bullet_\infty\otimes_R\oL)) \xleftarrow[\cong]{K_\infty \hookleftarrow R_\infty} 
H^i(C^\bullet(U, R_\infty\otimes_R\oL)) \cong 
H^i(U, M\otimes_R\oL) \xrightarrow[\cong]{R_\infty\hookrightarrow K_\infty}
\bH^i(U, I^\bullet_\infty \otimes_R\oL).
\]
We obtain the following diagram:
\begin{equation}\label{eq:peach}
\begin{tikzcd}[column sep = 6em, row sep = 1.5em]    \bH^i(U,\cF^\bullet_m)
    \arrow[r, "\wt\Xi"']
    \arrow[d, "(g\otimes a{,}b)\mapsto a(b) \cdot g"]
    &
        H^{i+1}(U, R_\infty \otimes_R \oL)
        \arrow[dd, "\cong"', "R_\infty \to I_\infty^\bullet"]
\\
    \Hom_{R_\infty}(H_i(U,R_\infty \otimes_R \cL), K_\infty/{R_\infty})
    \arrow[d, "K_\infty/{R_\infty} = I_\infty^1"]
        \\
H^i(\Hom_{R_\infty}^\bullet(C_\bullet(U, R_\infty \otimes_R \cL), I_\infty^\bullet[1]))
\arrow[r, "\text{Theorem~\ref{thm:dualHomology}~\eqref{part:perfectPairing}}", "\sim"']
&
    \bH^i(U, I^\bullet_\infty[1] \otimes_R \oL) 
    .
\end{tikzcd}
\end{equation}
We can explicitly write both compositions in this diagram, at the level of chain complexes. The right hand path is simply induced by the map of sheaves that vanishes in degree \(0\) and in degree \(-1\) it is the map \(R_\infty \otimes_R \oL \to K_\infty \otimes_R \oL\) induced by \(-\Id_{R_\infty}\).

The left hand path factors through the map \(\cF^\bullet_m\to \cF^0_m\) induced by the identity of \(\cF^0\). Let us write the corresponding map of (co)chain complexes on any degree \(i\)
\[
\begin{tikzcd}[column sep = 1em, row sep = 1em]
  &
\Hom_{R_\infty}^i(C_\bullet(U,R_\infty\otimes_R \cL), I_\infty^\bullet[1]) 
\arrow[from = d, rotate = 90, "\subset"']
\arrow[r]
  &
  C^\bullet(U, I^\bullet_\infty[1]\otimes_R \oL)
\arrow[from = d, rotate = 90, "\subset"']
  \\
    C^i(U,\cF^0_m)
    \arrow[r]
    &
\Hom_{R_\infty}(C_i(U,R_\infty\otimes_R \cL), I_\infty^0[1]) 
\arrow[r]
&
    C^i(U, I_\infty^0[1]\otimes_{R}\oL )
    \\
    (\gamma \mapsto g\otimes a)
    \arrow[r, mapsto] & 
    (\gamma \otimes (h\otimes b) \mapsto a(b)\cdot h \cdot g)
    \arrow[r, mapsto] & 
    (\gamma \mapsto g\otimes a).
\end{tikzcd}
\]
It is the map in cohomology induced by the map of complexes of sheaves \(\cF^\bullet_m\to I^\bullet_\infty [1]\otimes_R \oL\) that vanishes in degree \(-1\) and is induced by the map \(s^{-m}R_{\infty}\to K_\infty / R_\infty\) in degree \(0\).

Finally, note that both paths in diagram~\eqref{eq:peach} we have computed are homotopic, so they induce the same map in cohomology. The homotopy is the inclusion of \(s^{-m}R_\infty \otimes_R \oL \to K_\infty \otimes_R \oL\):
\[
\begin{tikzcd}[column sep = 7em]
    R_\infty \otimes_R \oL
    \arrow[r, hookrightarrow]
    \arrow[d, shift left = 0.7ex, "-\Id_{K_\infty}"]
    \arrow[d, shift right = 0.7ex, "0"']
    &
    s^{-m}R_\infty \otimes_R \oL
    \arrow[d, shift left = 0.7ex, "0"]
    \arrow[d, shift right = 0.7ex, "\Id_{K_\infty}"']
    \arrow[dl, "\Id_{K_\infty}"']
    \\
    K_\infty \otimes_R \oL
    \arrow[r, twoheadrightarrow]
    &
    K_\infty/R_\infty \otimes_R \oL.
\end{tikzcd}
\]

  \end{proof}

Our goal now is to show that all the morphisms from diagram~\eqref{eq:bigComposition} in Lemma~\ref{lem:commutes} except for $\Xi$ are MHS morphisms. We do this in a series of lemmas.

\begin{lemma}\label{lem:quotientSideMHS}
    Let $m\geq 1$. Let $\Theta^\vee$ be as in Definition~\ref{def:endowedMHSHomology} and let $p$ be as in Definition~\ref{def:p}. The map
   $$p\circ\Theta^\vee\colon H^i(U, R_{-m}\otimes_R \oL) \to \Hom_\Q\left(\frac{H_i(U^f, \Q)}{(t-1)^m H_i(U^f, \Q)}, \Q\right)$$
    is a MHS morphism for all $i$, where the MHS on \(H^i(U, R_{-m}\otimes_R \oL)\) is given by Definition \ref{def:endowedMHS}, and the (dual) MHS on the target is the one given by Definition-Proposition~\ref{defprop:MHSquotientsN}.
\end{lemma}
\begin{proof}
    Note that $\Theta^\vee$ is a MHS morphism by Definition~\ref{def:endowedMHSHomology}. Therefore, it suffices to show that the dual map \(
    \frac{H_i(U^f, \Q)}{(t-1)^mH_i(U^f, \Q)} \to H_i(U, R_m\otimes_R \cL)
    \) to $p$ is a MHS morphism. Note that, by definition of $p$, this dual map is induced by the projection $\cL\twoheadrightarrow R_m\otimes_R\cL$.

    The canonical isomorphism \( \frac{H_i(U,\cL)}{(t-1)^mH_i(U, \cL)}
        \cong
        \frac{H_i(U,R_\infty \otimes_R \cL)}{(t-1)^mH_i(U, R_\infty \otimes_R \cL)}\) is a MHS isomorphism, by construction (\cite[Corollary 6.16]{abeliancovers}). Now, the morphism induced by the projection $R_\infty\otimes_R\cL\twoheadrightarrow R_m\otimes_R\cL$ in homology factors as:
    \[
        H_i(U, R_\infty \otimes \cL)        
        \twoheadrightarrow
        \frac{H_i(U,R_\infty \otimes_R \cL)}{(t-1)^mH_i(U, R_\infty \otimes_R \cL)}
        \to
        H_i(U, R_m\otimes \cL),
    \]
    where the second map is the one that we need to show is a MHS morphism.
    
    The composition is a MHS morphism by the definition of the pro-MHS on its domain, in Definition~\ref{def:proMHS}. The first arrow is a quotient MHS by the construction in \cite[Corollary 6.15]{abeliancovers}. Therefore, the map we are interested in must also be a MHS morphism, since the Hodge and weight filtrations are strict (\cite[Corollary 3.6]{peters2008mixed}).
\end{proof}

\begin{lemma}\label{lem:pairingMHS}
    The map
$$\pairing_m\colon\mathbb{H}^i(U,\cF_m^\bullet) \to H^i(U,R_{-m}\otimes_R\oL)$$
   defined in Lemma~\ref{lem:commutes} is a MHS morphism for all $i$, where the MHS on \(H^i(U, R_{-m}\otimes_R \oL)\) is given by Definition~\ref{def:endowedMHS}, and the MHS on $\mathbb{H}^i(U,\cF_m^\bullet)$ is defined in Definition-Proposition~\ref{defprop:aux}.
\end{lemma}

\begin{proof}
    Note that the map $\pairing_m\colon\cF_m^\bullet \to R_{-m}\otimes_R\oL$ factors as composition of the quotient $\cF_m^\bullet\to \wt\cF^\bullet_{m,0}$ and the map $\pairing_m\colon\wt\cF^\bullet_{m,0}=\frac{s^{-m}R_\infty}{R_\infty}\otimes_R\oL\to R_{-m} \otimes_R\oL$ induced by the pairing $\pairing_m$ from Remark~\ref{rem:RmDual}. The quotient map induces an isomorphism of MHS in cohomology by Definition-Proposition~\ref{defprop:aux}~\eqref{part:defMHSF}. Hence, it suffices to show that the morphism induced by $\pairing_m\colon\wt\cF^\bullet_{m,0}=\frac{s^{-m}R_\infty}{R_\infty}\otimes_R\oL\to R_{-m} \otimes_R\oL$ in cohomology is a MHS morphism. We will do this by showing that $\pairing_m$ is induced by the rational part of a morphism of mixed Hodge complexes, namely by \(\bD_m\colon 
      \sC_{m,0}^\bullet = \sK^\bullet_{m}(1-m) \xrightarrow{\cong}
      \sK^\bullet_{-m}
    \) as defined in Theorem~\ref{thm:MHC}~\eqref{part:dual}. In other words, we need to see that the following diagram commutes:
    \[
      \begin{tikzcd}[row sep = 1.2em]
        \wt \cF^\bullet_{m,0} 
      \arrow[d, "\cong"', "\cV_{m,0}"]
      \arrow[r, equals]
      &
      \frac{s^{-m}R_\infty}{R_\infty} \otimes_R \oL
      \arrow[d , "\cong"', "\nu_{m}"]
      \arrow[r, "\langle\cdot {,}\cdot\rangle_m", "\cong"']
      &
      R_{-m}\otimes_R \oL
      \arrow[d, "\nu_{-m}", "\cong"']
      \\
      j^{-1}\sC^\bullet_{m,0}
      \arrow[r, equals]
      &
      j^{-1}\sK^\bullet_{m}(1-m)
      \arrow[r, "\cong"', "\bD_m"]
      &
      j^{-1}\sK^\bullet_{-m}
    \end{tikzcd}
    \]
    where $\cV_{m,0}$ is defined in Definition-Proposition~\ref{defprop:aux}~\eqref{part:defNu0}. Indeed, the diagram commutes, by Theorem~\ref{thm:MHC}~\eqref{part:dual}.
\end{proof}

\begin{lemma}\label{lem:XiMHS}
    Let $m\geq 1$. The map
$$\wt\Xi\colon\cF_m^\bullet \to R_\infty\otimes_R\oL[1]$$ from Definition~\ref{def:wtXi} induces (pro)-MHS morphisms in cohomology
$$
\wt\Xi\colon\mathbb{H}^i(U,\cF_m^\bullet) \to R_\infty\otimes_RH^{i+1}(U,\oL)
$$
for all $i$, where the MHS on $\mathbb{H}^i(U,\cF_m^\bullet)$ is defined in Definition-Proposition~\ref{defprop:aux} and the pro-MHS on  $R_\infty\otimes_RH^{i+1}(U,\oL)$ is defined in Definition~\ref{def:proMHS}.
\end{lemma}

\begin{proof}
MHS morphisms between two given $\Q$-MHS form a \(\Q\)-vector space, so we can equivalently show that \(-\wt \Xi\) induces a pro-MHS morphism. Let $c: \sC^\bullet_{m,m'} \to \sK^\bullet_{m'}[1]$ be the morphism of mixed Hodge complexes from Remark~\ref{rem:lesMixedCone} for all $m'\geq 0$.

Let us define the morphism
    $$
    \wt \Xi_{m'} \colon \wt\cF^\bullet_{m,m'}\to R_{m'}\otimes_R\cL[1]
    $$
    as the identity in degree $-1$ and vanishing in degree $0$. Since the diagram
    $$
    \begin{tikzcd}
    \cF^\bullet_m\arrow[r, "-\wt \Xi "]\arrow[d, two heads,"\text{projection}"] & R_{\infty}\otimes_R\cL[1]\arrow[d, two heads, "\text{projection}"]\\
    \wt\cF^\bullet_{m,m'}\arrow[r, "\wt \Xi_{m'}"] & R_{m'}\otimes_R\cL[1]
    \end{tikzcd}
    $$
    commutes, it suffices to show that $\wt \Xi_{m'}$ is a MHS morphism for all $m'\geq 1$, where the MHS on the cohomology of $R_{m'}\otimes_R\oL$ is the one from Definition~\ref{def:endowedMHS}. Indeed, $\mathbb{H}^i(U,\cF_m^\bullet)$ is finite dimensional, so its image through \(\wt \Xi \) is also finite dimensional. This amounts to showing that the diagram

    $$
    \begin{tikzcd}
       \wt\cF^\bullet_{m,m'}\arrow[r, "\wt \Xi_{m'}"]\arrow[d, "\cV_{m,m'}"] & R_{m'}\otimes_R\cL[1]\arrow[d, "\nu_{m'}"]\\
       \sC_{m,m'}^\bullet \arrow[r, "c"]& \sK_{m'}[1]
    \end{tikzcd}
    $$
    commutes, which is a straightforward verification in degrees $-1$ and $0$ which follows immediately from applying the definition of $\mathcal{V}_{m,m'}$ in Definition-Proposition~\ref{defprop:aux}~\eqref{part:defNu}, the definition of $c$ from Remark~\ref{rem:lesMixedCone} and the definition of $\wt \Xi_{m'}$.
\end{proof}

\begin{lemma}\label{lem:bottomInclusion}
 The inclusion
$$\Tors_R H^{i+1}(U,\oL)_1\to R_\infty\otimes_R H^{i+1}(U,\oL)$$
   is a MHS morphism for all $i$, where the MHS on the domain (resp. the target) is given in Theorem \ref{thm:tors1cohomology} (resp. Definition \ref{def:proMHS}).
\end{lemma}
\begin{proof}
By Definition \ref{def:proMHS} of the pro-MHS on \(R_\infty\otimes_R H^{i+1}(U,\oL)\) as an inverse limit, it suffices to show that for any large enough \(m'\), \(\Tors_R H^{i+1}(U,\oL)_1\to H^{i+1}(U,R_{m'}\otimes_R \oL)\) is a MHS morphism. This is the content of Theorem \ref{thm:tors1cohomology}.

\end{proof}

  We are now ready to prove the main theorem of this paper.
  \begin{proof}[Proof of Theorem~\ref{thm:main}]
  First, by Proposition-Definition~\ref{defprop:MHSquotients}
, all the quotient maps of the form
\[
  \displaystyle\frac{H_i(U^f, \Q)}{(t-1)^{m'}H_i(U^f, \Q)}\to \frac{H_i(U^f, \Q)}{(t-1)^mH_i(U^f, \Q)}
\]
are MHS morphisms. Therefore, we may assume that \(m\gg 1\).

  By Theorem~\ref{thm:eigenspaces},  \(\TorsH\) is a direct sum of MHS:
  \[
  \TorsH = \TorsH_1 \oplus
  \Big(
  \bigoplus_{\lambda\neq 1} \TorsH_\lambda
  \Big).
  \]
  Note that the restriction of the composition~\eqref{eq:mainThm} to the second summand vanishes. Therefore, it suffices to show that the restriction to \( \TorsH_1 \to \frac{\Alex}{(t-1)^m\Alex}
  \) is a MHS morphism. Equivalently, it suffices to show that its dual \(\Xi\) is a MHS morphism. By Lemma~\ref{lem:commutes}, it suffices to show that every other map except for \(\Xi\) in the commutative diagram~\eqref{eq:bigComposition} is a MHS morphism, which follows from Definition~\ref{def:endowedMHSHomology},  Lemmas~\ref{lem:quotientSideMHS}, \ref{lem:pairingMHS}, \ref{lem:XiMHS}, and \ref{lem:bottomInclusion} and Theorem~\ref{thm:tors1homology}.
\end{proof}

\begin{corollary}\label{cor:mainN}
  Let \(U\), \(f\) and \(U^f\) be as in Definition \ref{def:Uf}. Let \(m, N\ge 1\). Let \(R = \Q[\pi_1(\C^*)]\cong \Q[t^{\pm 1}]\). Consider the MHS on \(\Tors_R H_i(U^f,\Q)\) from Definition~\ref{def:mhsalexander}, and the MHS on \(\frac{H_i(U^f, \Q)}{(t^N-1)^m H_i(U^f, \Q)}\) from Definition-Proposition~\ref{defprop:MHSquotientsN}. Since they are a submodule and a quotient of the Alexander module, respectively, there is a natural composition map:
  \[
    \Tors_R H_i(U^f, \Q) \hookrightarrow H_i(U^f, \Q) \twoheadrightarrow \frac{H_i(U^f, \Q)}{(t^N-1)^mH_i(U^f, \Q)}.
  \]
  This composition map is a MHS morphism for all $m, N\geq 1$. Moreover, there exists $N\in \Z_{\geq 1}$ such that this composition map is injective for $m\gg 1$.
\end{corollary}
\begin{proof}
The ``moreover'' part of the statement is a consequence of the first sentence in Construction~\ref{con:mhsalexander}.

Now, let $N\geq 1$, and notice that the previous sentence implies that there exists $k\geq 1$ such that $\Tors_R H_i(U^f,\Q)$ is annihilated by a big enough power of $t^{kN}-1$ for all $i\geq 0$. Let $M=kN$. Consider the degree $M$ covering space $U_M$ of $U$, the map $f_M:U_M\to\C^*$ and the isomorphism of complex analytic varieties $\theta_M:(U_M)^{f_M}\to U^f$ as in Construction~\ref{con:mhsalexander}. We have he following commutative diagram:
    $$
    \begin{tikzcd}
      \Tors_{\Q[t^{\pm M}]} H_i\left((U_M)^{f_M},\Q\right)_1\arrow[r,hook]\arrow[dd,"\cong", "\theta_M"']&H_i\left((U_M)^{f_M},\Q\right)\arrow[r,two heads]\arrow[dd,"\cong", "\theta_M"'] & \frac{H_i\left((U_M)^{f_M},\Q\right)}{(t^M-1)^mH_i\left((U_M)^{f_M},\Q\right)}\arrow[d,two heads, , "\theta_M"']\\
      & & \frac{H_i(U^f, \Q)}{(t^M-1)^m H_i(U^f, \Q)}\arrow[d,two heads, "\text{projection}"]\\
       \Tors_RH_i(U^f, \Q)\arrow[r, hook] & H_i(U^f, \Q) \arrow[r, two heads]& \frac{H_i(U^f, \Q)}{(t^N-1)^m H_i(U^f, \Q)}
    \end{tikzcd}
    $$
    The composition of the maps in the top row is a MHS morphism by Theorem~\ref{thm:main}. The vertical arrow on the left is a MHS isomorphism by the definition of the MHS on $\Tors_RH_i(U^f, \Q)$ in \cite{mhsalexander} (see Construction \ref{con:mhsalexander}). Similarly, the vertical arrow induced by $\theta_M$ on the right is a MHS morphism by definition of the MHS on $\frac{H_i(U^f, \Q)}{(t^M-1)^m H_i(U^f, \Q)}$ from Definition-Proposition~\ref{defprop:MHSquotientsN} (made explicit in \cite[Proposition 8.3]{abeliancovers}). The arrow labeled by ``projection'' is a MHS morphism by Definition-Proposition~\ref{defprop:MHSquotientsN} as well. The result now follows from the commutativity of the diagram.
\end{proof}

\bibliographystyle{plain}
\bibliography{Bibliography}
\end{document}